\documentclass[12pt]{amsart}
\usepackage{graphicx}
\usepackage{amsmath}
\usepackage{subfigure}
\usepackage{amssymb} 
\usepackage[usenames,dvipsnames]{color}
\usepackage{pinlabel}
\usepackage{mathtools}
\usepackage{float}
\usepackage{tikz}
  \usetikzlibrary{hobby}
  \usetikzlibrary{patterns}

\newtheorem{thm}{Theorem}[section]  
\newtheorem{cor}[thm]{Corollary}

\newtheorem{defin}[thm]{Definition} 
\newtheorem{lemma}[thm]{Lemma} 
 
\newtheorem{prop}[thm]{Proposition} 
 
\newtheorem{ass}[thm]{Assumption} 
\newtheorem*{defin*}{Definition}
\newcommand{\aaa}{\mbox{$\alpha$}}
\newcommand{\bbb}{\mbox{$\beta$}}
 
\newcommand{\Ddd}{\mbox{$\Delta$}}

\newcommand{\Ggg}{\mbox{$\Gamma$}}

\newcommand{\calA}{\mbox{$\mathcal A$}}
\newcommand{\calD}{\mbox{$\mathcal D$}}
\newcommand{\calE}{\mbox{$\mathcal E$}}
\newcommand{\sss}{\mbox{$\sigma$}}  
\newcommand{\Sss}{\mbox{$\Sigma$}}

\newcommand{\bdd}{\mbox{$\partial$}}
\newcommand{\inter}{\mbox{${\rm int}$}}
\def\real{{\mathbb R}}
\newcommand{\fra}{\mbox{$\mathfrak{a}$}}

\newcommand{\Rrr}{\mbox{$\mathfrak{R}$}}
\newcommand{\Mrr}{\mbox{$M_{\Rrr}$}}

\begin{document}  

\title{A strong Haken's Theorem}

\author{Martin Scharlemann}
\address{\hskip-\parindent
        Martin Scharlemann\\
        Mathematics Department\\
        University of California\\
        Santa Barbara, CA 93106-3080 USA}
\email{mgscharl@math.ucsb.edu}

\date{\today}

\begin{abstract} Suppose $M = A \cup_T B$ is a Heegaard split compact orientable $3$-manifold and $S \subset M$ is a reducing sphere for $M$.  Haken \cite{Ha} showed  that there is then also a reducing sphere $S^*$ for the Heegaard splitting.  Casson-Gordon \cite{CG} extended the result to $\bdd$-reducing disks in $M$ and noted that in both cases $S^*$ is obtained from $S$ by a sequence of operations called $1$-surgeries.  Here we show that in fact one may take $S^* = S$.   \end{abstract}

\maketitle



It is a foundational theorem of Haken \cite{Ha} that any Heegaard splitting $M = A \cup_T B$ of a closed orientable reducible $3$-manifold $M$ is reducible; that is, there is an essential sphere in the manifold that intersects $T$ in a single circle.  Casson-Gordon \cite[Lemma 1.1]{CG} refined and generalized the theorem, showing that it applies also to essential disks, when $M$ has boundary. More specifically, if $S$ is a disjoint union of essential disks and $2$-spheres in $M$ then there is a similar family $S^*$, obtained from $S$ by ambient $1$-surgery and isotopy, so that each component of $S^*$ intersects $T$ in a single circle.  In particular, if $M$ is irreducible, so $S$ consists entirely of disks, $S^*$ is isotopic to $S$.  

There is of course a more natural statement, in which $S$ does not have to be replaced by $S^*$.  I became interested in whether the natural statement is true because it would be the first step in a program to characterize generators of the Goeritz group of $S^3$, see \cite{FS2}, \cite{Sc2}.  Inquiring of experts, I learned that this more natural statement had been pursued by some, but not successfully.  Here we present such a proof.  A reader who would like to get the main idea in a short amount of time could start with the example in Section \ref{section:Zupan}.  Recently Hensel and Schultens \cite{HS} have proposed an alternate proof that applies when $M$ is closed and $S$ consists entirely of spheres.  

Here is an outline of the paper:  Sections \ref{sect:intro} and \ref{sect:vertical} are mostly a review of what is known, particularly the use of verticality in classical compression bodies, those which have no spheres in their boundary.  We wish to allow sphere components in the boundary, and Section \ref{sect:vertical2} explains how to recover the classical results in this context.  Section \ref{sect:connreduce} shows how to use these results to inductively reduce the proof of the main theorem to the case when $S$ is connected.  The proof when $S$ is connected (the core of the proof) then occupies Sections \ref{sect:early} through \ref{sect:final}.  

\section{Introduction and Review} \label{sect:intro}

All manifolds considered will be orientable and, unless otherwise described, also compact.  For $M$ a $3$-manifold, a closed surface $T \subset M$ is a {\em Heegaard surface} in $M$ if the closed complementary components $A$ and $B$ are each compression bodies, defined below. This structure is called a {\em Heegaard splitting} and is typically written $M = A \cup_T B$.  See, for example, \cite{Sc1} for an overview of the general theory of Heegaard surfaces.  Among the foundational theorems of the subject is this \cite{CG}:

 Suppose $T$ is a Heegaard surface in a Heegaard split $3$-manifold 
$M = A \cup_T B$ and $D$ is a $\bdd$ reducing disk for $M$, with $\bdd D \subset \bdd_- B \subset \bdd M$.  

\begin{thm}[Haken, Casson-Gordon]  There is a $\bdd$-reducing disk $E$ for $M$ such that
\begin{itemize}
\item $\bdd E = \bdd D$
\item $E$ intersects $T$ in a single essential circle (i. e. $E$ $ \bdd$-reduces $T$)
\end{itemize}
\end{thm}

Note: $D$ and $E$ are isotopic if $M$ is irreducible; but if $M$ is reducible then there is no claim that $D$ and $E$ are isotopic.  

There is a similar foundational theorem, by Haken alone \cite{Ha}, that  if $M$ is reducible, there is a reducing sphere for $M$ that intersects $T$ in a single circle (i. e. it is a reducing sphere for $T$).  But Haken made no claim that the reducing sphere for $T$ is isotopic to a given reducing sphere for $M$.  

The intention of this paper is to fill this gap in our understanding.  We begin by retreating to a more general setting.  For our purposes, a {\em compression body} $C$ is a connected $3$-manifold obtained from a (typically disconnected) closed surface $\bdd_- C$ by attaching $1$-handles to one end of a collar of $\bdd_- C$.  The closed connected surface $\bdd C - \bdd_- C$ is denoted $\bdd_+ C$. This differs from what may be the standard notion in that we allow $\bdd_- C$ to contain spheres, so $C$ may be reducible.  Put another way, we take the standard notion, but then allow the compression body to be punctured finitely many times.  In particular, the compact $3$-manifolds whose Heegaard splittings we study may have spheres as boundary components.  

Suppose then that $M = A \cup_T B$ is a Heegaard splitting, with $A$ and $B$ compression bodies as above.  A disk/sphere set $(S, \bdd S) \subset (M, \bdd M)$ is a properly embedded surface in $M$ so that each component of $S$  is either a disk or a sphere. A sphere in $M$ is called {\em inessential} if it either bounds a ball or is parallel to a boundary component of $M$; a disk is inessential if it is parallel to a disk in $\bdd M$.  $S$ may contain such inessential components, but these are easily dismissed, as we will see.  

\begin{defin} The Heegaard splitting $T$ is {\em aligned} with $S$ (or vice versa) if each component of $S$ intersects $T$ in at most one circle.
\end{defin} 
For example, a reducing sphere or $\bdd$-reducing disk for $T$, typically defined as a sphere or disk that intersects $T$ in a single essential circle, are each important examples of an aligned disk/sphere.  This new terminology is introduced in part because, in the mathematical context of this paper, the word 'reduce' is used in multiple ways that can be confusing.  More importantly, once we generalize compression bodies as above, so that some boundary components may be spheres, there are essential spheres and disks in $M$ that may miss $T$ entirely and others that may intersect $T$ only in curves that are inessential in $T$.  We need to take these disks and spheres into account.

\begin{thm} \label{thm:main}
Suppose that $(S, \bdd S) \subset (M, \bdd M)$ is a disk/sphere set in $M$.  Then there is an isotopy of $T$ so that afterwards $T$ is aligned with $S$.

Moreover, such an isotopy can be found so that, after the alignment, the annular components $S \cap A$, if any, form a vertical family of spanning annuli in the compression body $A$, and similarly for $S \cap B$.
 \end{thm}
 
The terminology ``vertical family of spanning annuli" is defined in Section \ref{sect:vertical}.  
  
Note that a disk/sphere set $S$ may contain inessential disks or spheres, or essential disks whose boundaries are inessential in $\bdd M$.  Each of these are examples in which the disk or sphere could lie entirely in one of the compression bodies and so be disjoint from $T$.  In the classical setting, Theorem \ref{thm:main} has this immediate corollary:

 \begin{cor}[Strong Haken] \label{cor:main}
 Suppose $\bdd M$ contains no sphere components.  Suppose $S \subset M$ (resp $(S, \bdd S) \subset  (M, \bdd M))$ is a reducing sphere (resp $\bdd$-reducing disk) in $M$.  Then $S$ is isotopic to a reducing sphere (resp $\bdd$-reducing disk) for $T$.
 \end{cor}
 
The assumption in Corollary \ref{cor:main} that there are no sphere components in $\bdd M$ puts us in the classical setting, where any reducing sphere $S$ for $M$ must intersect $T$.  

\section{Verticality in aspherical compression bodies} \label{sect:vertical}

We first briefly review  some classic facts and terminology for an aspherical compression body $C$, by which we mean that $\bdd_- C$ contains no sphere components.  Later, sphere components will add a small but interesting amount of complexity to this standard theory.  See \cite{Sc1} for a fuller account of the classical theory.  Unstated in that account (and others) is the following elementary observation, which further supports the use of the term ``aspherical":

\begin{prop} \label{prop:comirred} An aspherical compression body $C$ is irreducible.
\end{prop}
\begin{proof} Let $\Delta$ be the cocores of the $1$-handles used in the construction of $C$ from the collar $\bdd_- C \times I$.  If $C$ contained a reducing sphere $S$, that is a sphere that does not bound a ball, a standard innermost disk argument on $S \cap \Delta$ would show that there is a reducing sphere in the collar $\bdd_- C \times I$.  But since $C$ is assumed to be aspherical, $\bdd_-C$ contains no spheres, and it is classical that a collar of a closed orientable surface that is not a sphere is irreducible.  (For example, it's universal cover is a collar of $R^2$; the interior of this collar is $R^3$; and $R^3$ is known to be irreducible by the Schonfliess Theorem \cite{Sch}. )
\end{proof}

\begin{defin} \label{defin:comcoll} A properly embedded family $(\Delta, \bdd \Delta) \subset (C, \bdd C)$ of disks is a {\em complete collection of meridian disks} for $C$ if $C - \eta(\Delta)$ consists of a collar of $\bdd_- C$ and some $3$-balls.
\end{defin}

That there is such a family of disks follows from the definition of a compression body: take $\Delta$ to be the cocores of the $1$-handles used in the construction.  Given two complete collections $\Delta, \Delta'$ of meridian disks in an aspherical compression body, it is possible to make them disjoint by a sequence of $2$-handle slides, viewing the disks as cocores of $2$-handles.  (The slides are often more easily viewed dually, as slides of $1$-handles.) The argument in brief is this: if $\Delta$ and $\Delta'$ are two complete collections of meridians, an innermost disk argument (which relies on asphericity) can be used to remove all circles of intersection.  A disk cut off from $\Delta'$ by an outermost arc $\gamma$ of $\Delta' \cap \Delta$ in $\Delta'$ determines a way of sliding the $2$-handle in $\Delta$ containing $\gamma$ over some other members of $\Delta$ to eliminate $\gamma$ without creating more intersection arcs. Continue until all arcs are gone. (A bit more detail is contained in Phase 2 of the proof of Proposition \ref{prop:snugdisjoint}.)

Visually, one can think of the cores of the balls and $1$-handles as a properly embedded graph in $C$, with some valence $1$ vertices on $\bdd_- C$, so that the union $\Sss$ of the graph and $\bdd_- C$ has $C$ as its regular neighborhood.  $\Sss$ is called a {\em spine} of the compression body.  As already noted, a spine for $C$ is far from unique, but one can move from any spine to any other spine by sliding ends of edges in the graph over other edges, or over components of $\bdd_- C$, dual to the $2$-handle slides described above.  (See \cite{ST} or \cite{Sc1}.)  For most arguments it is sufficient and also simplifying to disregard any valence one vertex  that is not on $\bdd_-C$ and the ``cancelling" edge to which it is attached (but these do briefly appear in the proof of Corollary \ref{cor:stemswap}); to disregard all valence 2 vertices by amalgamating the incident edges into a single edge; and, via a slight perturbation, to require all vertices not on $\bdd_- C$ to be of valence 3.    We can, by edge slides, ensure that only a single edge of the spine is incident to each component of $\bdd_- C$; this choice of spine is also sometimes useful.

The spine can be defined as above even when $\bdd_- C$ contains spheres.
Figure \ref{fig:snug} shows a schematic picture of a (non-aspherical) compression body, viewed first with its (aqua) two-handle structure and then its dual $1$-handle (spinal) structure.  $\bdd_- C$ is the union of a torus and 3 spheres; the genus two $\bdd_+ C$ appears in the spinal diagram only as an imagined boundary of a regular neighborhood of the spine.

    \begin{figure}[th]
    \centering
    \includegraphics[scale=0.6]{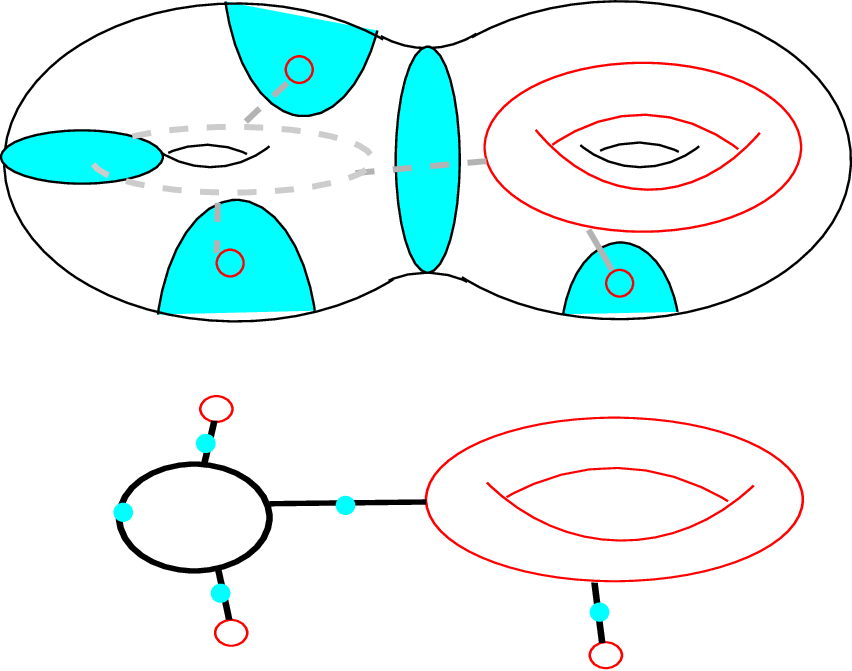}
\caption{2-handles and dual spine in a compression body.} 
 \label{fig:snug}
    \end{figure}
 
\begin{defin} \label{defin:vertical} A properly embedded arc $\alpha$ in a compression body $C$ is {\em spanning} if one end of $\alpha$ lies on each of $\bdd_- C$ and $\bdd_+ C$.  Similarly, a properly embedded annulus in $C$ is {\em spanning} if one end lies in each of $\bdd_- C$ and $\bdd_+ C$.  (Hence each spanning arc in a spanning annulus is also spanning in the compression body.)  

A disjoint collection of spanning arcs $\alpha$ in a compression body
is a {\em vertical family of arcs} if there is a complete collection $\Delta$ of meridian disks for $C$ so that 
\begin{itemize}
\item $\alpha \cap \Delta = \emptyset$ and
\item for $N$, the components of $C - \Delta$ that are a collar of $\bdd_- C$, there is a homeomorphism $h: \bdd_- C \times (I, \{0\}) \to (N, \bdd_- C)$ so that $h(\mathfrak{p} \times I) = \alpha$, where $\mathfrak{p}$ is a collection of points in $\bdd_-C$.
\end{itemize}  
\end{defin}

A word of caution: we will show in Proposition \ref{prop:vertunique} that any two vertical arcs with endpoints on the same component $F \subset \bdd_-C$ are properly isotopic in $C$.  This is obvious if the two constitute a vertical family.   If they are each vertical, but not as a vertical family, proof is required because the collection of meridian disks referred to in Definition \ref{defin:vertical} may differ for the two arcs.  

There is a relatively simple but quite useful way of characterizing a vertical family of arcs.   To that end, let $\alpha$ be a family of spanning arcs in $C$ and  $\hat{p} = \alpha \cap \bdd_- C$ be their end points in $\bdd_- C$.   An embedded family $c$ of simple closed curves in $\bdd_-C$ is a {\em circle family associated to $\alpha$} if 
$\hat{p} \subset c$. 

\begin{lemma} \label{lemma:annulus} Suppose $\alpha$ is a family of spanning arcs in an aspherical compression body $C$.
\begin{itemize}
\item Suppose $\alpha$ is vertical and $c$ is an associated circle family.  Then there is a family $\calA$ of disjoint spanning annuli in $C$ so that $\calA$ contains $\alpha$ and $\calA \cap \bdd_-C = c.$
\item Suppose, on the other hand, there is a collection $\calA$ of disjoint spanning annuli in $C$ that contains $\alpha$. Suppose further that in the family of circles $\calA \cap \bdd_-C$ associated to $\alpha$, each circle is essential in $\bdd_- C$. 
Then $\alpha$ is a vertical family.  
\end{itemize}
\end{lemma}

    \begin{figure}[th]
     \labellist
   \small\hair 2pt
\pinlabel  $p_A$ at 15 235
\pinlabel  $\aaa \cap A$ at 120 190
\pinlabel  $\aaa'$ at 70 245
\pinlabel  $\Delta \cap A$ at 80 80
\pinlabel  $\Delta \cap A$ at 270 120
\endlabellist
    \centering
    \includegraphics[scale=0.6]{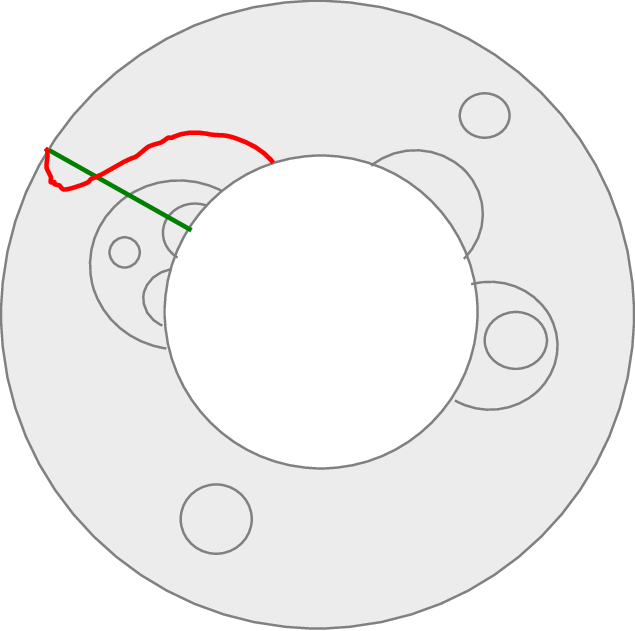}
\caption{$\alpha'$ avoids $\Delta \cap A$} 
 \label{fig:annulus}
    \end{figure}

\begin{proof} One direction is clear: Suppose  $\alpha$ is a vertical family and $h: \bdd_-C \times (I, \{0\}) \to (N, \bdd_- C)$ is the homeomorphism from Definition \ref{defin:vertical}.  Then $h(c \times I)$ is the required family of spanning annuli.  (After the technical adjustment, from general position, of moving the circles $h(c \times \{1\})$ off the disks in $h(\bdd_-C \times \{1\})$ coming from the family $\Delta$ of meridian disks for $C$.)

For the second claim, 
let $\Delta$ be any complete collection of meridians for $C$ and consider the collection of curves $\Delta \cap \calA$.  If $\Delta \cap \calA = \emptyset$ then $\calA$ is a family of incompressible spanning annuli in the collar $\bdd_- C \times I$ and, by standard arguments, any family of incompressible spanning annuli in a collar is vertical.  Furthermore, any family of spanning arcs in a vertical annulus can visibly be isotoped rel one end of the annulus to be a family of vertical arcs.  So we are left with the case $\Delta \cap \calA \neq \emptyset$. 

Suppose $\Delta \cap \calA$ contains a simple closed curve, necessarily inessential in $\Delta$.  If that curve were essential in a component $A \in \calA$, then the end $A \cap \bdd_-C \subset c$ would be null-homotopic in $C$.  Since the hypothesis is that each such circle is essential in $\bdd_- C$, this would contradict the injectivity of $\pi_1(\bdd_- C) \to \pi_1(C)$. 

We conclude that each component of $\Delta \cap \calA$ is either an inessential circle in $\calA$ or an arc in $\calA$ with both ends on $\bdd_+ C$, since $\bdd \Delta \subset \bdd_+C$. Such arcs are inessential in $\calA$.   

Consider what this means in a component $A \in \calA$; let $c_A = A \cap \bdd_-C \in c$ be the end of $A$ in $\bdd_-C$.  It is easy to find spanning arcs $\alpha'$ in $A$ with ends at the points $p_A = \hat{p} \cap c_A$, chosen so that $\alpha'$ avoids all components of $\Delta \cap A$. See Figure \ref{fig:annulus}. But, as spanning arcs, $\alpha \cap A$ and $\alpha'$ are isotopic in $A$ rel $c_A$ (or, if one prefers, one can picture this as an isotopy near $A$ that moves the curves $\Delta \cap A$ off of $\alpha \cap A$).    After such an isotopy in each annulus, $\Delta$ and  $\alpha$ are disjoint.  Now apply classic innermost disk, outermost arc arguments to alter $\Delta$ until it becomes a complete collection of meridians disjoint from $\calA$, the case we have already considered.  More details of this classic argument appear in Phase 2 of the proof of Proposition \ref{prop:snugdisjoint}.
\end{proof}

Lemma \ref{lemma:annulus} suggests the following definition.  

\begin{defin} \label{defin:vertann} Suppose $\calA $ is a family of disjoint spanning annuli in $C$ and $\alpha$ is a collection of disjoint spanning arcs in $\calA$, with at least one arc  of $\alpha$ in each annulus of $\calA$.  $\calA$ is a {\em vertical family of annuli} if and only if $\alpha$ is a vertical family of arcs.  
\end{defin}

Note that for $\calA$ to be vertical we do not require that $\calA$ be incompressible in $C$.   This adds some complexity to our later arguments, particularly the proof of Proposition \ref{prop:vertdodge}.  

\begin{prop} \label{prop:annulusdisjoint} Suppose $\calA$ is a vertical family of annuli in an aspherical compression body $C$.  Then there is a complete collection of meridian disks for $C$ that is disjoint from $\calA$.
\end{prop}

\begin{proof} Let $\alpha \subset \calA$ be a vertical family of spanning arcs as given in Definition \ref{defin:vertann}.  Since $\alpha$ is a vertical family of arcs, there is a complete collection $\Delta$ of meridian disks for $C$ that is disjoint from $\alpha$, so $\Delta$ intersects $\calA$ only in inessential circles, and arcs with both ends incident to the end of $\bdd A$ at $\bdd_+ C$.  As noted in the proof of Lemma \ref{lemma:annulus}, a standard innermost disk, outermost arc argument can be used to alter $\Delta$ to be disjoint from $\calA$.  
\end{proof}

\begin{cor} \label{cor:annulidiskdis} Suppose $(\calD, \bdd \calD) \subset (C, \bdd_+C)$ is an embedded  family of disks that is disjoint from an embedded family of vertical annuli $\calA$ in an aspherical compression body $C$.
Then there is a complete collection of meridian disks for $C$ that is disjoint from $\calA \cup \calD$.
\end{cor}

\begin{proof} Propositions \ref{prop:annulusdisjoint} 
shows that there is a complete collection disjoint from $\calA$.  But the same proof (which exploits asphericity through its use of Lemma \ref{lemma:annulus}) works here, if we augment the curves $\Delta \cap \calA$ with also the circles $\Delta \cap \calD$.
\end{proof}

\begin{prop} \label{prop:vertunique} Suppose $F$ is a component of $\bdd_- C$ and $\alpha, \beta$ are vertical arcs  in $C$ with endpoints $p, q \in F$.
Then $\alpha$ and $\beta$ are properly isotopic in $C$.  
\end{prop} 

Notice that the proposition does not claim that $\alpha$ and $ \beta$ are parallel, so in particular they do not necessarily constitute a vertical family.  Indeed the isotopy from $\alpha$ to $\beta$ that we will describe may involve crossings between $\alpha$ and $\beta$.

    \begin{figure}[th]
     \labellist
   \small\hair 2pt
\pinlabel  $q$ at 270 50
\pinlabel  $\bbb$ at 290 65
\pinlabel  $\aaa$ at 35 155
\pinlabel  $\gamma$ at 95 180
\pinlabel  $\gamma$ at 380 180
\pinlabel  $p$ at 10 160
\pinlabel  $c_\alpha$ at 70 210
\pinlabel  $c_\beta$ at 340 210
\pinlabel  $A_\alpha$ at 40 70
\pinlabel  $A_\beta$ at 300 140
\pinlabel  $A_\alpha \cap A_\beta$ at 180 80
\pinlabel  $A_\alpha \cap A_\beta$ at 450 100
\endlabellist
    \centering
    \includegraphics[scale=0.7]{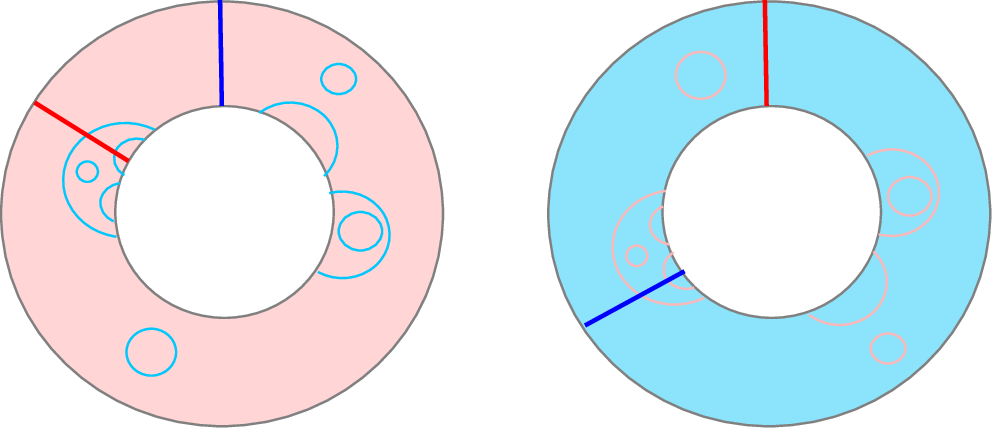}
\caption{Arcs $\alpha$ and $\beta$ both properly isotopic to $\gamma$} 
 \label{fig:vertunique}
    \end{figure}

\begin{proof}  Since $C$ is aspherical, $genus(F) \geq 1$ and there are simple closed curves $c_{\aaa}, c_{\bbb} \subset F$ so that
\begin{itemize}
\item $p \in c_{\aaa}, q \in c_{\bbb}$
\item $c_{\aaa}$ and $c_{\bbb}$ intersect in a single point.
\end{itemize}

Since $\aaa$ and $\bbb$ are each vertical, it follows from Lemma \ref{lemma:annulus} that there are spanning annuli $A_{\aaa}, A_{\bbb}$ in $C$ that contain $\aaa$ and $\bbb$ respectively and whose ends on $F$ are $c_{\aaa}$ and $c_{\bbb}$ respectively.  Since  $c_{\aaa}$ and $c_{\bbb}$ intersect in a single point, this means that  among the curves in $A_{\aaa} \cap A_{\bbb}$ there is a single arc $\gamma$ that spans each annulus, and no other arcs are incident to $F$.  The annulus $A_{\aaa}$ then provides a proper isotopy from the spanning arc $\aaa$ to $\gamma$ and the annulus $A_{\bbb}$ provides a proper isotopy from $\gamma$ to $\bbb$.  Hence $\alpha$ and $\beta$ are properly isotopic in $C$. See Figure \ref{fig:vertunique}.  
\end{proof}

\bigskip

We now embark on a technical lemma that uses these ideas, a lemma that we will need later.   Begin with a closed connected surface $F$ that is not a sphere, and say that circles $\alpha, \beta$ {\em essentially intersect} if they are not isotopic to disjoint circles and have been isotoped so that $|\alpha \cap \beta|$ is minimized.  Suppose $\hat{a} \subset F$ is an embedded family of simple closed curves, not necessarily essential, and $p_1, p_2$ is a pair of points disjoint from $\hat{a}$.  (We only will need the case of two points; the argument below extends to any finite number, with some loss of clarity in statement and proof.)

Let $b' \subset F$ be a non-separating simple closed curve in $F$ that is not parallel to any $a \in \hat{a}$.  For example, if all curves in $\hat{a}$ are separating, $b'$ could be any non-separating curve; if some curve $a \in \hat{a}$ is non-separating, take $b'$ to be a circle that intersects $a$ once.  Isotope $b'$ in $F$ so that it contains $p_1, p_2$, and intersects $\hat{a}$ transversally if at all; call the result $b \subset F$.  (Note that, following these requirements, $\hat{a}$ may not intersect $b$ essentially, for example if an innermost disk in $F$ cut off by an inessential $a \in \hat{a}$ contains $p_i$.) If $b$ intersects $\hat{a}$, let $q_i$ be points in $b \cap \hat{a}$ so that the subintervals $\sss_i \subset b$ between $p_i$ and $q_i$ have interiors disjoint from $\hat{a}$ and are also disjoint from each other.  Informally, we could say that $q_i$ is the closest point in $\hat{a}$ to $p_i$ along $b$, and $\sigma_i$ is the path in $b$ between $p_i$ and $q_i$.  

Since $b$ is non-separating there is a simple closed curve $x \subset F$ that intersects $b$ exactly twice, with the same orientation (so the intersection is essential).  Isotope $x$ along $b$ until the two points of intersection are exactly $q_1, q_2$.  See Figure \ref{fig:vertdodge1}.

    \begin{figure}[th]
     \labellist
   \small\hair 2pt
\pinlabel  $a$ at 50 140
\pinlabel  $a$ at 45 70
\pinlabel  $\sss$ at 70 63
\pinlabel  $x$ at 85 40
\pinlabel  $b$ at 100 70
\pinlabel  $q$ at 85 55
\pinlabel  $p$ at 60 65
\endlabellist
    \centering
    \includegraphics[scale=1.0]{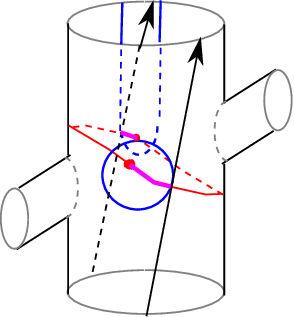}
\caption{Preamble to Lemma \ref{lemma:vertdodge}} 
 \label{fig:vertdodge1}
    \end{figure}

 \begin{lemma} \label{lemma:vertdodge}  Let $(\calD, \bdd \calD) \subset (C, \bdd_+ C)$ and $\calA \subset C$ be as in Corollary \ref{cor:annulidiskdis}.  Suppose $\hat{\beta} = \{\bbb_i\}, i = 1, 2$ is a vertical family of arcs in $C$ whose end points $p_i \in \bdd_- C$ are disjoint from the family of circles $\hat{a} =  \calA \cap \bdd_- C$ in $\bdd_- C$.  Then $\hat{\beta}$ can be properly isotoped rel  $\{p_i\}$ so that it is disjoint from $\calA \cup \calD$.
\end{lemma}

\begin{proof} We suppose that both components of $\hat{\beta}$ are incident to the same component $F$ of $\bdd_- C$.  The proof is essentially the same (indeed easier) if they are incident to different components of $\bdd_- C$.  Let $\Delta$ be a complete family of meridian disks as given in Corollary \ref{cor:annulidiskdis}, so $\calA$ lies entirely in a collar of $\bdd_- C$.  Per Lemma \ref{lemma:annulus}, let $B \subset C$ be a spanning annulus that contains the vertical pair $\hat{\beta}$ and has the curve $b$ (from the preamble to this lemma) as its end $B \cap F$ on $F$.  

Suppose first that $b$ is disjoint from $\hat{a}$ and consider $B \cap (\Delta \cup \calD \cup \calA)$.  If there were a circle $c$ of intersection that is essential in $B$, then it could not be in $\Delta \cup \calD$, since $b$ does not compress in $C$.  The circle $c$ could not be essential in $\calA$, since $b$ was chosen so that it is not isotopic to any element of $\hat{a}$, and it can't be inessential there either again since $b$ does not compress in $C$.  We deduce that there can be no essential circle of intersection, so any circles in $B \cap (\Delta \cup \calD \cup \calA)$ are inessential in $B$.  Also, any arc of intersection must have both ends on $\bdd_+ C$ since $b$ is disjoint from $\hat{a}$.  It follows that the spanning arcs $\hat{\beta}$ of $B$ can be properly isotoped in $B$ to arcs that avoid $\Delta \cup \calD \cup \calA$.  So, note, they are in the collar of $\bdd_- C$ as well as being disjoint from $\calA \cup \calD$ as required.

    \begin{figure}[th]
     \labellist
   \small\hair 2pt
\pinlabel  $B$ at 300 65
\pinlabel  $\aaa$ at 115 165
\pinlabel  $\gamma$ at 140 170
\pinlabel  $x$ at 10 160
\pinlabel  $b$ at 280 180
\pinlabel  $q$ at 100 210
\pinlabel  $p$ at 365 215
\pinlabel  $\beta$ at 360 165
\pinlabel  $\sss$ at 390 210
\pinlabel  $q$ at 410 210
\pinlabel  $X$ at 40 70
\pinlabel  $\gamma$ at 300 160
\pinlabel  $B\cap\Delta$ at 450 100
\endlabellist
    \centering
    \includegraphics[scale=0.7]{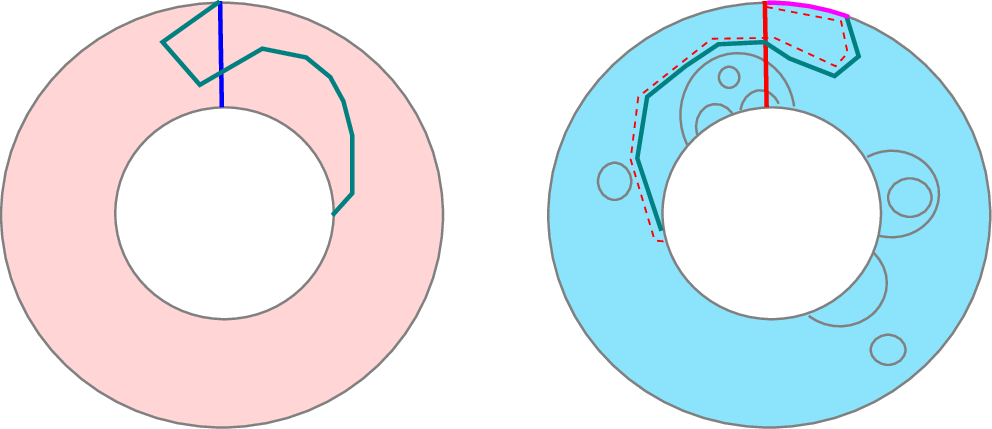}
\caption{Concluding the proof of Lemma \ref{lemma:vertdodge}} 
 \label{fig:vertdodge2}
    \end{figure}

Now suppose that $b$ is not disjoint from $\hat{a}$ and let the points $q_i$, the subarcs $\sss_i$ of $b$ and the simple closed curve $x \subset F$ be as described in the preamble to this lemma.  By construction, each $q_i$ is in the end of an annulus $A_i \subset \calA$; let $\alpha_i \subset A_i$ be a spanning arc of $A_i$ with an end on $q_i$.  Since $\calA$ is a vertical family of annuli, $\alpha_1, \alpha_2$ is a vertical pair of spanning arcs.  Per Lemma \ref{lemma:annulus},  there is a spanning annulus $X$ that contains the $\alpha_i$ and has the curve $x$ as its end $X \cap F$ on $F$.  Since $x$ essentially intersects $b$ in these two points, $B \cap X$ contains exactly two spanning arcs $\gamma_i, i = 1, 2$, each with one end point on the respective $q_i$.  

In $B$ the spanning arcs $\bbb_i$ can be properly isotoped rel $p_i$ so that they are each very near the concatenation of $\sss_i$ and $\gamma_i$; in $X$ the arcs $\gamma_i$ can be properly isotoped rel $q_i$ to $\alpha_i$. See Figure \ref{fig:vertdodge2}.   (One could also think of this as giving an ambient isotopy of the annulus $B$ so that afterwards $\gamma_i = \alpha_i$.) The combination of these isotopies then leaves $\beta_i$ parallel to the arc $\sss_i \cup \alpha_i$.  A slight push-off away from $A_i$ leaves $\beta_i$ disjoint from $\calA \cup \calD$ as required.
\end{proof}

\section{Verticality in compression bodies} \label{sect:vertical2}

We no longer will assume that compression bodies are aspherical.  That is, $\bdd_- C$ may contain spheres.  We will denote by $\hat{C}$ the aspherical compression body obtained by attaching a $3$-ball to each such sphere.

Figure \ref{fig:snug} shows a particularly useful type of meridian disk to consider when $\bdd_- C$ contains spheres.

\begin{defin}  \label{defin:snug}
A complete collection $\Delta$ of meridian disks in a compression body $C$ is a {\em snug collection} if, for each sphere $F \subset \bdd_- C$ the associated collar of $F$ in $C - \Delta$ is incident to exactly one disk $D_F \in \Delta$.
\end{defin}

The use of the word ``snug" is motivated by a simple construction.  Suppose $\Delta$ is a snug collection of meridian disks for $C$ and $F \subset \bdd_- C$ is a sphere.  Then the associated disk $D_F \subset \Delta$ is completely determined by a spanning arc $\alpha_{F}$ in the collar of $F$ in $C - \Delta$, and vice versa:  The arc $\alpha_{F}$ is uniquely determined by $D_F$, by the light-bulb trick, and once $\alpha_F$ is given, $D_F$ is recovered simply by taking a regular neighborhood of $\alpha_F \cup F$.  This regular neighborhood is a collar of $F$, and the end of the collar away from $F$ itself is the boundary union of a disk in $\bdd_+ C$ and a copy of $D_F$.  With that description, we picture $D_F$ as sitting ``snugly" around  $\alpha_{F} \cup F$.  See Figure \ref{fig:snug3}.

    \begin{figure}[th]
     \labellist
   \small\hair 2pt
\pinlabel  $D_F$ at 150 70
\pinlabel  $\alpha_F$ at 120 13
\pinlabel  $\bdd_+C$ at 190 20
\pinlabel  $F$ at 110 35
\endlabellist
    \centering
    \includegraphics[scale=0.8]{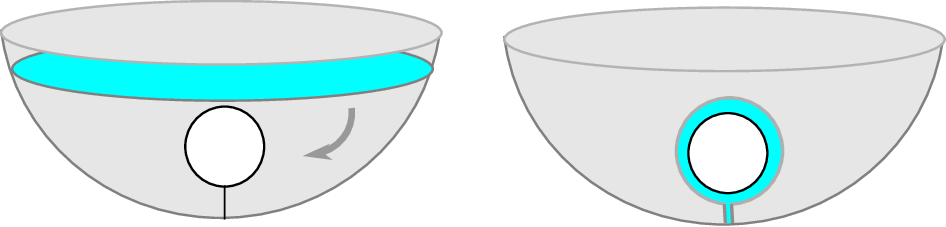}
\caption{$D_F$ snuggles down around $\alpha_{F} \cup F$} 
 \label{fig:snug3}
    \end{figure}

Following immediately from Definition \ref{defin:snug} is:

\begin{lemma} \label{lemma:fillsnug} Suppose $C$ is a compression body and $\hat{\Delta}$ is a collection of meridian disks for $C$ that is a complete collection for the aspherical compression body $\hat{C}$.  Then $\hat{\Delta}$ is contained in a snug collection for $C$.
\end{lemma}

\begin{proof}  
For each sphere component $F$ of $\bdd_- C$, let $\alpha_F$ be a properly embedded arc in $C - \hat{\Delta}$ from $F$ to $\bdd_- C$ and construct a corresponding meridian disk $D_F$ as just described.  Then the union of $\hat{\Delta}$ with all these new meridian disks is a snug collection for $C$.
\end{proof}

Following Definition \ref{defin:comcoll} we noted that for an aspherical compression body, two complete collections of meridian disks can be handle-slid and isotoped to be disjoint.  As a useful warm-up we will show that this is also true for snug collections, in case $\bdd_- C$ contains spheres.  This is the key lemma:

\begin{lemma} \label{lemma:Zupan} Suppose $C$ is a compression-body with $p, q \in \bdd_+ C$ and $r \in interior(C)$.  Suppose $\aaa$ and $\bbb$ are arcs from $p$ and $q$ respectively to $r$ in $C$.  Then there is a proper isotopy of $\beta$ to $\alpha$ in $C$, fixing $r$. 
 \end{lemma}  

    \begin{figure}[th]
     \labellist
   \small\hair 2pt
\pinlabel  $q$ at 220 40
\pinlabel  $p$ at 40 0
\pinlabel  $r$ at 170 90
\pinlabel  $\aaa$ at 120 110
\pinlabel  $\bbb$ at 340 110
\pinlabel  $\gamma$ at 140 145
\endlabellist
    \centering
    \includegraphics[scale=0.6]{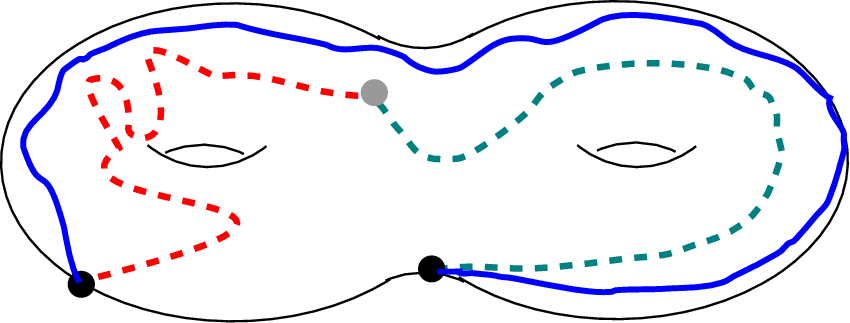}
\caption{} 
 \label{fig:ZupanRev}
    \end{figure}
 
\begin{proof}  Let $\Sss$ be a spine for the compression-body $C$.  By general position, we may take $\Sss$ to be disjoint from  the path $\aaa \cup \bbb$.  Since $\pi_1(\bdd_+ C) \to \pi_1(C)$ is surjective there is a path $\gamma$ in $\bdd C$ so that the closed curve $\aaa \cup \bbb \cup \gamma$ is null-homotopic in $C$.  Slide the end of $\bbb$ at $q$ along $\gamma$ to $p$ so that $\beta$ becomes an arc $\bbb'$ (parallel to the concatenation of $\gamma$ and $\bbb$) also from $p$ to $r$, one that is homotopic to $\aaa$ rel end points. A sophisticated version of the light-bulb trick (\cite[Proposition 4]{HT}) then shows that $\aaa$ and $\bbb'$ are isotopic rel end points.  (Early versions of this paper appealed to the far more complex  \cite[Theorem 0]{FS1} to provide such an isotopy.)
 \end{proof}   
 
 \begin{prop} \label{prop:snugdisjoint} Suppose $\Delta$ and $\Delta'$ are snug collections of meridian disks for $C$.  Then $\Delta$ can be made disjoint from $\Delta'$ by a sequence of handle slides and proper isotopies.
 \end{prop} 
 
 \begin{proof}  Let $\mathcal{F} = \{F_i\}, 1 \leq i \leq n$ be the collection of spherical boundary components of $C$.  Since $\Delta$ (resp $\Delta'$) is snug, to each $F_i$ there corresponds a properly embedded arc $\alpha_i$ (resp $\alpha'_i$) in $C$ from $F_i$ to $\bdd_+ C$ and this arc determines the meridian disk in $D_i \subset \Delta$ (resp $D'_i \subset \Delta'$) associated to $F_i$ as described after Definition \ref{defin:snug}.  The proof in the aspherical case (as outlined following Definition \ref{defin:comcoll}; see also \cite{Sc1}) was achieved by isotopies and slides reducing $|\Delta \cap \Delta'|$.  In the general case the proof proceeds in two phases. 
 
 {\bf Phase 1:}  We will properly isotope the arcs $\{\alpha_i\}, 1 \leq i \leq n$ to $\{\alpha'_i\}, 1 \leq i \leq n$.  The associated ambient isotopy of $\Delta$ in $C$ may increase $|\Delta \cap \Delta'|$ but in this first phase we don't care.  Once each $\alpha_i = \alpha_i'$, each snug disk $D_i$ can be made parallel to $D'_i$ by construction.
 
Pick a sphere component $F_i$ with associated arcs $\alpha_i$ and $\alpha'_i$.  Isotope the end of $\alpha_i$ on $F_i$ to the end $r$ of $\alpha'_i$ at $F_i$.  Temporarily attach a ball $B$ to $F_i$ and apply Lemma \ref{lemma:Zupan} to the arcs $\alpha, \alpha'$, after which $\alpha$ and $\alpha'$ coincide.  By general position, we can assume the isotopy misses the center $b$ of $B$ and by the light-bulb trick that it never passes through the radius of $B$ between $b$ and $r$.  Now use radial projection from $b$ to push the isotopy entirely out of $B$ and thus back into $C$.

Having established how to do the isotopy for a single $\alpha_i$, observe that we can perform such an isotopy simultaneously on all $\alpha_i, 1 \leq i \leq n$.  Indeed, anytime the isotopy of $\alpha_i$ is to cross $\alpha_j, i \neq j$ we can avoid the crossing by pushing it along $\alpha_j$, over the sphere $F_j$, and then back along $\alpha_j$; in short, use the light-bulb trick.

{\bf Phase 2:} We eliminate $\Delta \cap \Delta'$ by reducing $|\Delta \cap \Delta'|$, as in the aspherical case.  After Phase 1, the disks $\{D_i\}, 1 \leq i \leq n$ are parallel to the disks $\{D'_i\}, 1 \leq i \leq n$; until the end of this phase we take them to coincide and also to be fixed, neither isotoped nor slid.  Denote the complement in $\Delta$ (resp $\Delta'$) of this collection of disks $\{D_i\}$ by $\hat{\Delta}$ (resp $\hat{\Delta}'$), since they constitute a complete collection of meridians in $\hat{C}$.  Moreover, the component of $C - \{D_i\}$ containing $\hat{\Delta}$ and $\hat{\Delta}'$ is homeomorphic to $\hat{C}$, so that is how we will designate that component.

Motivated by that last observation, we now complete the proof by isotoping and sliding  $\hat{\Delta}$, much as in the aspherical case, to reduce $|\hat{\Delta} \cap \hat{ \Delta'}|$. Suppose first there are circles of intersection and let $E' \subset \hat{ \Delta'}$ be a disk with interior disjoint from $\hat{\Delta}$ cut off by an innermost such circle of intersection in $\hat{ \Delta'}$.  Then $\bdd E'$ also bounds a disk $E \subset \hat{\Delta}$ (which may further intersect $\hat{ \Delta'}$).  Although $C$ is no longer aspherical, the sphere $E \cup E'$ lies entirely in $\hat{C}$, which is aspherical, so $E \cup E'$ bounds a ball in $\hat{C}$, through which we can isotope $E$ past $E'$, reducing $|\hat{\Delta} \cap \hat{ \Delta'}|$ by at least one.

Once all the circles of intersection are eliminated as described, we consider arcs in $\hat{\Delta} \cap \hat{ \Delta'}$.  An outermost such arc in $\hat{ \Delta'}$ cuts off a disk $E'$ from $\hat{ \Delta'}$ that is disjoint from $\hat{\Delta}$; the same arc cuts off a disk $E$ from $\hat{\Delta}$ (which may further intersect $\hat{ \Delta'}$). The properly embedded disk $E \cup E' \subset \hat{C}$ has boundary on $\bdd_+ \hat{C}$ and its interior is disjoint from $\Delta$.  The latter fact means that its boundary lies on one end of the collar $\hat{C} - \eta(\Delta)$ of a non-spherical component $F$ of $\bdd_- C$.  But in a collar of $F$ any properly embedded disk is $\bdd$-parallel.  Use the disk in the end of the collar (the other end from $F$ itself) to which $E \cup E'$ is parallel to slide $E$ past $E'$ (possibly sliding it over other disks in $\Delta$, including those in $\{D_i\}$), thereby reducing  $|\hat{\Delta} \cap \hat{\Delta'}|$ by at least one.

Once $\hat{\Delta}$ and $\hat{\Delta'}$ are disjoint, slightly push the disks $\{D_i\}$ off the presently coinciding disks $\{D'_i\}$ so that $\Delta$ and $\Delta'$ are disjoint.  
\end{proof}

Energized by these observations we will now show that all the results of Section \ref{sect:vertical} remain true (in an appropriate form) in compression bodies that are not aspherical, that is, even when there are sphere components of $\bdd_- C$.  Here are the analogous results, with edits on statement in boldface, and proofs annotated as appropriate:

\begin{lemma}[cf. Lemma \ref{lemma:annulus}] \label{lemma:sannulus} Suppose $\hat{\alpha}$ is a family of spanning arcs in compression body $C$.

\begin{itemize}
\item Suppose $\hat{\alpha}$ is vertical and $c$ is an associated circle family.  Then there is a family $\calA$ of disjoint spanning annuli in $C$ so that $\calA$ contains $\hat{\alpha}$ and $\calA \cap \bdd_-C = c.$
\item  Suppose, on the other hand, there is a collection $\calA$ of disjoint spanning annuli in $C$ that contains $\hat{\alpha}$.  Suppose further that
\begin{itemize} 
\item {\bf \em at most one arc in $\hat{\alpha}$ is incident to each sphere component of $\bdd_-C$} and
\item in the family of circles $\calA \cap \bdd_-C$ associated to $\hat{\alpha}$, each circle  {\bf \em lying in a non-spherical component of $\bdd_- C$} is essential. 
\end{itemize}  
Then $\alpha$ is a vertical family.  
\end{itemize}
\end{lemma}

\begin{proof} The proof of the first statement is unchanged.  

For the second, observe that by Lemma \ref{lemma:annulus} there is a collection $\hat{\Delta}$ of meridian disks in $\hat{C}$ so that $\hat{\Delta}$ is disjoint from each arc $\alpha \in \hat{\alpha}$ that is incident to a non-spherical component of $\bdd_- C$.  By general position, $\hat{\Delta}$ can be taken to be disjoint from the balls $C - \hat{C}$ and so lie in $C$.  

Now consider an arc $\aaa' \in \hat{\alpha}$ that is incident to a sphere $F$ in $\bdd_- C$.  It may be that $\hat{\Delta}$ intersects $\aaa'$.  In this case, push a neighborhood of each point of intersection along $\alpha'$ and then over $F$.  Note that this last operation is not an isotopy of $\hat{\Delta}$ in $C$, since it pops across $F$, but that's unimportant - afterwards the (new) $\hat{\Delta}$ is completely disjoint from $\aaa'$.  Repeat the operation for every component of $\hat{\alpha}$ that is incident to a sphere in $\bdd_- C$, so that $\hat{\Delta}$ is disjoint from all of $\hat{\alpha}$.  Now apply the proof of Lemma \ref{lemma:fillsnug}, expanding $\hat{\Delta}$ by adding a snug meridian disk for each sphere in $\bdd_-C$, using the corresponding arc in $\hat{\alpha}$ to define the snug meridian disk for spheres that are incident to $\hat{\alpha}$.
\end{proof}

\begin{prop}[cf. Corollary \ref{cor:annulidiskdis}] \label{prop:saannulisdiskdis} Suppose $(\calD, \bdd \calD) \subset (C, \bdd_+C)$ is an embedded  family of disks that is disjoint from an embedded family of vertical annuli $\calA$ in $C$. Then there is a complete 
collection of meridian disks for $C$ that is disjoint from $\calA \cup \calD$.
\end{prop}

\begin{proof} Let $\alpha \subset \calA$ be a vertical family of spanning arcs as given in Definition \ref{defin:vertann}.  This means there is a complete collection $\Delta$ of meridian disks for $C$ that is disjoint from $\alpha$, so $\Delta$ intersects $\calA$ only in inessential circles, and in arcs with both ends incident to the end of $\calA$ at $\bdd_+ C$.  

Let $C'$ be the compression body obtained by attaching a ball to each sphere component of $\bdd_- C$ that is {\em not incident to $\calA$}.  Because $\Delta$ is a complete collection in $C$, it is also a complete collection in $C'$, since attaching a ball to a collar of a sphere just creates a ball.  Consider the curves $\Delta \cap (\calA \cup \calD)$, and proceed as usual, much as in Phase 2 of the proof of Proposition \ref{prop:snugdisjoint}:  

If there are circles of intersection, an innermost one in $\Delta$ cuts off a disk $E \subset \Delta$ and a disk $E' \subset (\calA \cup \calD)$ which together form a sphere whose interior is disjoint from $\calA$ and so bounds a ball in $C'$.  In $C'$, $E'$ can be isotoped across $E$, reducing $|\Delta \cap (\calA \cup \calD)|$.  On the other hand, if there are no circles of intersecton, then an arc of intersection $\gamma$ outermost in $\calA \cup \calD$ cuts off a disk $E' \subset (\calA \cup \calD)$ and a disk $E \subset \Delta$ which together form a properly embedded disk $E''$ in $C' - \Delta$ whose boundary lies on $\bdd_+ C$.  Since $E''$ lies in $C' - \Delta$ it lies in a collar of $\bdd_- C'$ and so is parallel to a disk in the other end of the collar.  (If the relevant component of $\bdd_- C'$ is a sphere, we may have to reset $E$ to be the other half of the disk in $\Delta$ in which $\gamma$ lies to accomplish this.) The disk allows us to slide $E$ past $E'$ and so reduce $|\Delta \cap (\calA \cup \calD)|$.  

The upshot is that eventually, with slides and isotopies, $\Delta$ can be made disjoint from $\Delta \cap (\calA \cup \calD)$ {\em in $C'$}.  The isotopies themselves can't be done in $C$, since sphere boundary components {\em disjoint from $\calA$} may get in the way, but the result of the isotopy shows how to alter $\Delta$ (not necessarily by isotopy) to a family of disks $\Delta'$ disjoint from $\calA \cup \calD$ that is complete in $C'$.  Now apply the argument of Lemma \ref{lemma:fillsnug}, adding a snug disk to $\Delta'$ for each sphere component of $\bdd_- C$ that was not incident to $\calA$ and so bounded a ball in $C'$. These additional snug disks, when added to $\Delta'$, create a complete collection of meridian disks for $C$ that is disjoint from $\calA \cup \calD$, as required.
\end{proof}

\begin{prop}[cf. Proposition \ref{prop:vertunique}] \label{prop:svertunique} Suppose $\alpha, \beta$ are vertical arcs  in $C$ with endpoints $p, q$ in a component $F \subset \bdd_-C$.
Then $\alpha$ and $\beta$ are properly isotopic in $C$.  
\end{prop} 

\begin{proof}  If $F$ is not a sphere, apply the argument of Proposition \ref{prop:vertunique}.  If $F$ is a sphere, apply Lemma \ref{lemma:Zupan}. 
\end{proof}

 \begin{prop}[cf. Lemma \ref{lemma:vertdodge}] \label{prop:vertdodge} Suppose $(\calD, \bdd \calD) \subset (C, \bdd_+C)$ is an embedded  family of disks that is disjoint from an embedded family of vertical annuli $\calA$ in $C$.  Suppose $\hat{\beta} = \{\bbb_i\}, i = 1, 2$ is a vertical family of arcs in $C$ whose end points $p_i \in \bdd_- C$ are disjoint from the family of circles $\hat{a} =  \calA \cap \bdd_- C$ in $\bdd_- C$.  Then $\bbb$ can be properly isotoped rel  $\{p_i\}$ so that it is disjoint from $\calA \cup \calD$.
\end{prop}

\begin{proof}  The proof, like the statement, is essentially identical to that of Lemma \ref{lemma:vertdodge}, with this alteration when $F \subset \bdd C_-$ is a sphere:  Use Lemma \ref{lemma:Zupan} to isotope the vertical (hence parallel) pair $\hat{\beta}$ rel $p_i$ until the arcs are parallel to the vertical family of spanning arcs of $\calA$ that are incident to $F$.  In particular, we can then take $\hat{\beta}$ to lie in the same collar $F \times I$ as $\calA$ does, and to be parallel to $\calA$ in that collar.  It is then a simple matter, as in the proof of Lemma \ref{lemma:vertdodge}, to isotope each arc in $\hat{\beta}$ rel $p_i$ very near to the concatenation of arcs $\sss_i$ disjoint from $\calA$ and arcs $\alpha_i$ in $\calA$ and, once so positioned, to push $\hat{\beta}$ off of $\calA \cup \calD$. \end{proof}

Let us now return to the world and language of Heegaard splittings with a lemma on verticality, closely related to $\bdd$-reduction of Heegaard splittings.  

 Suppose $M = A \cup_T B$ is a Heegaard splitting of a compact orientable $3$-manifold $M$ and $(E, \bdd E) \subset (M, \bdd_-B)$ is a properly embedded disk, intersecting $T$ in a single circle, so that the annulus $E \cap B$ is vertical in $B$ and the disk $E \cap A$ is essential in $A$.  Since $E \cap B$ is vertical, there is a complete collection of meridian disks $\Delta$ in the compression body $B$ so that a component $N$ of $B - \Delta$ is a collar of $\bdd_- B$ in which $E \cap B$ is a vertical annulus.  Parameterize $E$ as a unit disk with center $b \in E \cap A$ and $E \cap B$ the set of points in $E$ with radius $\frac12 \leq r \leq 1$.  Let $\rho$ be a vertical radius of $E$, with $\rho_A$ the half in the disk $E \cap A$ and $\rho_B$ the half in the annulus $E \cap B$.  

Let $E \times [-1, 1]$ be a collar of the disk $E$ in $M$ and consider the manifold $M_0= M - (E \times (-\epsilon, \epsilon))$, the complement of a thinner collar of $E$.  It has a natural Heegaard splitting, obtained by moving the solid cylinders $(E \cap A) \times (-1, - \epsilon]$ and $(E \cap A) \times [\epsilon, -1)$ from $A$ to $B$.  Classically, this operation (when $E$ is essential) is called $\bdd$-reducing $T$ along $E$ \cite[Definition 3.5]{Sc1}.  We denote this splitting by $M_0 = A_0 \cup_{T_0} B_0$, recognizing that if $E$ is separating, it describes a Heegaard splitting of each component.   Denote the spanning arcs $b \times [-1, -\epsilon]$ and $b \times [\epsilon, 1]$ in $B_0$ by $\beta_-$ and $\beta_+$ respectively.  See the top two panes of Figure \ref{fig:vertbeta}, with a schematic rendering below.

    \begin{figure}[th]
     \labellist
   \small\hair 2pt
\pinlabel  $T$ at 175 175
\pinlabel  $T_0$ at 230 175
\pinlabel  $E$ at 90 110
\pinlabel  $N\subset B$ at 20 110
\pinlabel  $A$ at 20 150
\pinlabel  $F$ at 20 202
\pinlabel  $F_0$ at 250 205
\pinlabel  $b$ at 90 140
\pinlabel  $M$ at 180 150
\pinlabel  $M_0$ at 230 150
\pinlabel  $\bbb_-$ at 80 27
\pinlabel  $\bbb_+$ at 110 27
\pinlabel  $\rho_B$ at 100 60
\pinlabel  $\rho_A$ at 97 160
\endlabellist
    \centering
    \includegraphics[scale=0.9]{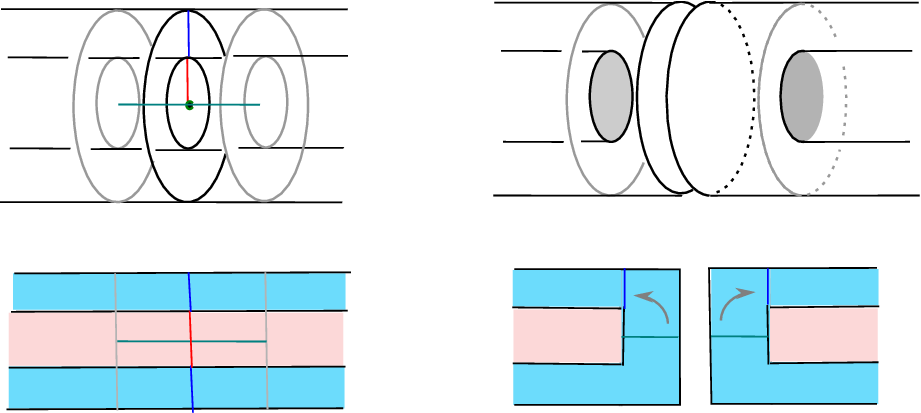}
\caption{} 
 \label{fig:vertbeta}
    \end{figure}

\begin{lemma} \label{lemma:vertbeta}
The spanning arcs $\beta_{\pm}$ are a vertical family of arcs in $B_0$.
\end{lemma}

\begin{proof} The complete collection of meridian disks $\Delta$ for $B$ is disjoint from the annulus $E \cap B$ so remains in $B_0$.  Viewed in the collar component $N \cong (F \times I)$ in the complement of $\Delta$ to which $E \cap B$ belongs, the operation described cuts the component $F \subset \bdd_- B$ by $\bdd E \subset F$, then caps off the boundary circles by disks to get a new surface $F_0$ and extends the collar structure to $F \times I$.  The rectangles $\rho \times [\epsilon, 1]$ and $\rho \times  [-1, -\epsilon]$ provide isotopies in $M_0$ from $\bbb_{\pm}$ to the vertical arcs $\rho_B \times \{\pm 1\}$, illustrating that $\bbb_{\pm}$ is a vertical family.
\end{proof}

\section{Reducing Theorem \ref{thm:main} to the case $S$ connected} \label{sect:connreduce}

To begin the proof of Theorem \ref{thm:main} note that (unsurprisingly) we may as well assume each component of $S$ is essential, that is no sphere in $S$ bounds a ball and no sphere or disk in $S$ is $\bdd$-parallel.  This can be accomplished simply by isotoping all inessential components well away from $T$.
So henceforth we will assume all components of $S$ are essential, including perhaps disks whose boundaries are inessential in $\bdd M$ but which are not $\bdd$-parallel in $M$.

Assign a simple notion of complexity $(g, s)$ to the pair $(M, T)$, with $g$ the genus of $T$ and $s$ the number of spherical boundary components of $M$.  We will induct on this pair, noting that there is nothing to prove if $g = 0$ and $s \leq 2$.  

Suppose then that we are given a disk/sphere set $(S, \bdd S) \subset (M, \bdd M)$ in which all components are essential. We begin with

\begin{ass}(Inductive assumption) \label{ass:induct} 
Theorem \ref{thm:main} is true for Heegaard splittings of manifolds that have lower complexity than that of $(M, T)$.
\end{ass}

With this inductive assumption we have:

\begin{prop}  \label{prop:S0} It suffices to prove Theorem \ref{thm:main} for a single component $S_0$ of $S$. 
\end{prop}

\begin{proof} Let $M = A \cup_T B$ be a Heegaard splitting, $S \subset M$ be a disk/sphere set, in which each component is essential in $M$, and let $S_0$ be a component of $S$ that is aligned with $T$. The goal is to isotope the other components of $S$ so that they are also aligned, using the inductive Assumption \ref{ass:induct}.

{\bf Case 1: $S_0$ is a sphere and $S_0 \cap T = \emptyset$ or an inessential curve in $T$.}

If $S_0$ is disjoint from $T$, say $S_0 \subset B$, then it cuts off from $M$ a punctured ball.  This follows from Proposition \ref{prop:comirred}, which shows that $S_0$ bounds a ball in the aspherical compression body $\hat{B}$ and so a punctured $3$-ball in $B$ itself. Any component of $S - S_0$ lying in the punctured $3$-ball is automatically aligned, since it is disjoint from $T$.  Removing the punctured $3$-ball from $B$ leaves a compression body $B_0$ with still at least one spherical boundary component, namely $S_0$.  The Heegaard split $M_0 = A \cup_T B_0$ is unchanged, except there are fewer boundary spheres in $B_0$ than in $B$ because $S_0$ is essential.  Now align all remaining components of $S - S_0$ using the inductive assumption, completing the construction.

Suppose next that $S_0$ intersects $T$ in a single circle that bounds a disk $D_T$ in $T$, and $S_0$ can't be isotoped off of $T$.  Then $S_0$ again bounds a punctured ball in $M$ with $m \geq 1$ spheres of $\bdd M$ lying in $A$ and $n \geq 1$ spheres of $\bdd M$ lying in $B$.  $S_0$ itself is cut by $T$ into hemispheres $D_A = S_0 \cap A$ and $D_B = S_0 \cap B$.  A useful picture can be obtained by regarding $D_A$ (say) as the cocore of a thin $1$-handle in $A$ connecting a copy $A_+$ of $A$ with $m$ fewer punctures to a boundary component $T_- = D_T \cup D_A$ of an $m$-punctured ball in $A$.  In this picture, $S_0$ and $T_-$ are parallel in $\hat{B}$; the interior of the collar between them has $n$ punctures in $B$ itself.  See Figure \ref{fig:propS01}.  

Let $\bbb$ be the core of the $1$-handle, divided by $S_0$ into a subarc $\bbb_+$ incident to $T_+ = \bdd A_+$ and $\bbb_-$ incident to the sphere $T_-$.  Now cut $M$ along $S_0$, dividing it into two pieces.  One is a copy $M_+ = A_+ \cup_{T_+} B_+$ of $M$, but with $m$ fewer punctures in $A_+$ and $n-1$ fewer in $B_+$ (a copy of $S_0$ is now a spherical boundary component of $B_+$).  The other is an $m + n +1$ punctured $3$-sphere $M_-$, Heegaard split by the sphere $T_-$.  (Neither of the spanning arcs $\bbb_+$ nor $\bbb_-$ play a role in these splittings yet.)

    \begin{figure}[th]
\labellist
\small\hair 2pt
\pinlabel  $A$ at 30 280
\pinlabel  $B$ at 30 200
\pinlabel  $T$ at 40 245
\pinlabel  $D_A$ at 140 290
\pinlabel  $D_B$ at 135 200
\pinlabel  $D_T$ at 100 245
\pinlabel  $T_-$ at 300 205
\pinlabel  $S_0$ at 345 205
\pinlabel  $A_+$ at 130 145
\pinlabel  $B_+$ at 130 90
\pinlabel  $T_+$ at 270 120
\pinlabel  $\bbb$ at 295 95
\pinlabel  $M_+$ at 10 100
\pinlabel  $M_-$ at 80 15
\pinlabel  $\bbb_+$ at 75 110
\pinlabel  $\bbb_-$ at 73 82

\endlabellist
    \centering
    \includegraphics[scale=0.8]{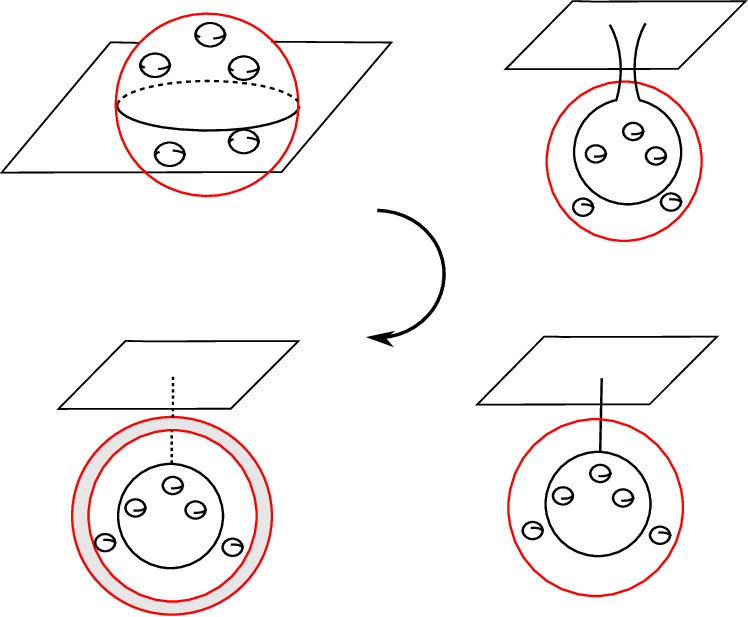}
\caption{Clockwise through inductive step in Case 1} 
 \label{fig:propS01}
    \end{figure}

Now apply the inductive assumption to align $T_+$ and $T_-$ with the disk/sphere set $S - S_0$ (not shown in Figure \ref{fig:propS01}).  Afterwards, reattach $M_+$ to $M_-$ along the copies of $S_0$ in each.  The result is again $M$, and $S$ is aligned with the two parts $T_-$ and $T_+$ in $T$.  But to recover $T$ itself, while ensuring that $S$ remains aligned, we need to ensure that $\beta$ can be properly isotoped rel $S_0$ so that it is disjoint from $S - S_0$. Such a proper isotopy of $\beta$ will determine an isotopy of $T$ by viewing $\beta$ as the core of a tube (the remaining part of $T$) connecting $T_+$ to $T_-$.   But once $S - S_0$ is aligned, the proper isotopy of $\beta$ can be found by first applying Proposition \ref{prop:vertdodge} to $\bbb_+$ and the family $S \cap B_0$ of disks and annuli in the compression body $B_+$ and then proceeding similarly with the arc $\bbb_-$ in $M_-$.  
\medskip

{\bf Case 2: $S_0$ is a sphere that intersects $T$ in an essential curve.}

As in Case 1, $S_0$ is cut by $T$ into hemispheres $D_A = S_0 \cap A$ and $D_B = S_0 \cap B$ and we can consider $D_A$ (say) as the cocore of a thin $1$-handle in $A$.  Continuing as in Case 1, denote the arc core of the $1$-handle by $\bbb$; $S_0$ again divides the arc $\bbb$ into two arcs which we label $\bbb_{\pm}$. 

If $S_0$ separates, then it divides $M$ into two manifolds, say $M_{\pm}$ containing respectively $\bbb_{\pm}$.   Apply the same argument in each that was applied in Case 1 to the manifold $M_+$.  

If $S_0$ is non-separating sphere then we can regard $S - S_0$ as a disk sphere set in the manifold $M_0 = M - \eta(S_0)$.  Since $S_0$ is two-sided, two copies $S_{\pm}$ of $S_0$ appear as spheres in $\bdd M_0$.  Choose the labelling so that each arc $\bbb_{\pm}$ has one end in the corresponding $S_{\pm}$.  $M_0$ has lower complexity (the genus is lower) so the inductive assumption applies, and the spheres in $S - S_0$ can be aligned with $T_0$.  Apply Proposition \ref{prop:vertdodge} to the arcs $\bbb_{\pm}$ and then reconstruct $(M, T)$, now with $T$ aligned with $S$, as in Case 1.
\medskip

{\bf Case 3: $S_0$ is a separating disk.} 

Suppose, with no loss of generality, that $\bdd S_0 \subset \bdd_- B$, so $S_0$ intersects $A$ in a separating disk $D_A$ and $B$ in a separating vertical spanning annulus. As in the previous cases, let $M_{\pm}$ be the manifolds obtained from $M$ by cutting along $S_0$, $\bbb$ the core of the $1$-handle in $A$ whose cocore is $D_A$ and $\bbb_{\pm}$ its two subarcs in $M_{\pm}$ respectively.  

The compression body $A - \eta(D_A)$ consists of two compression bodies, $A_{\pm}$ in $M_{\pm}$ respectively.  As described in the preamble to Lemma \ref{lemma:vertbeta}, the complement $B_{\pm}$ of $A_{\pm}$ in $M_{\pm}$ is a compression body, in which $\bbb_{\pm}$ is a vertical spanning arc.  So the surfaces $T_{\pm}$ obtained from $T$ by compressing along $D_A$ are Heegaard splitting surfaces for $M_{\pm}$, and the pairs $(M_{\pm}, T_{\pm})$ have lower complexity than $(M, T)$.

Now apply the inductive hypothesis: Isotope each of $T_{\pm}$ in $M_{\pm}$ so that they align with the components of $S - S_0$ lying in $M_{\pm}$.  As in Case 1, apply Proposition \ref{prop:vertdodge} to each of $\beta_{\pm}$ and then reattach $M_+$ to $M_-$ along disks in $\bdd M_{\pm}$ centered on the points $\bbb_{\pm} \cap \bdd M_{\pm}$ and simultaneously reattach $\bbb_+$ to $\bbb_-$ at those points.  The result is an arc isotopic to $\bbb$ which is disjoint from $S - S_0$.  Moreover, the original Heegaard surface $T$ can be recovered from $T_{\pm}$ by tubing them together along $\bbb$ and, since $\bbb$ is now disjoint from $S - S_0$, all of $T$ is aligned with $S$.
\medskip

{\bf Case 4: $S_0$ is a non-separating disk.}   
 
Near $S_0$ the argument is the same as in Case 3.  Now, however, the manifold $M_0$ obtained by cutting along $S_0$ is connected.  The construction of its Heegaard splitting $M_0 = A_0 \cup_{T_0} B_0$ and vertical spanning arcs $\bbb_{\pm}$ proceeds as in Case 3, and, since $genus(T_0) = genus(T) -1$, we can again apply the inductive hypothesis to align $S - S_0$ with $T_0$.  

If $\bdd S_0$ separates the component $F$ of $\bdd_- B \subset \bdd M$ in which it lies, say into surfaces $F_{\pm}$ the argument concludes just as in Case 3.  If $\bdd S_0$ is non-separating in $F$, then we encounter the technical point that Proposition \ref{prop:vertdodge} requires that $\beta$ be a vertical family of arcs. But this follows from Lemma \ref{lemma:vertbeta}.
\end{proof}

\section{Breaking symmetry: stem swaps} \label{sect:stemswap}

Applications of Lemma \ref{lemma:Zupan} extend beyond Propositions \ref{prop:vertunique} and \ref{prop:vertdodge}.  But the arguments will require {\bf breaking symmetry}: Given a Heegaard splitting $M = A \cup_T B$ of a compact orientable $3$-manifold $M$ and $\Sss$ a spine for $B$, we can, and typically will, regard $B$ as a thin regular neighborhood of $\Sss$, with $T$ as the boundary of that thin regular neighborhood.  This allows general position to be invoked as if $B$ were a graph embedded in $M$.  Edge slides of $\Sss$ can be viewed as isotopies of $T$ in $M$ and therefore typically are of little consequence.  We have encountered this idea in the previous section: the boundary of a tubular neighborhood of an arc $\bbb$ there represented an annulus in $T$; a proper isotopy of $\bbb$ was there interpreted as an isotopy of $T$.  We can then regard $A$ as the closure of $M - \eta(\Sss)$; a properly embedded arc in $A$ then appears as an arc whose interior lies in $M - \Sss$ and whose end points may be incident to $\Sss$.  We describe such an arc as a properly embedded arc in $A$ whose end points lie on $\Sss$.  This point of view is crucial to what follows; without it many of the statements might appear to be nonsense.  

Let $R$ be a sphere component of $\bdd_- B$.  Let $\Sss$ be a spine for $B$ for which a single edge $\sss$ is incident to $R$. 

\begin{defin} \label{defin:stem}
The complex $\sss \cup R$ is called a {\em flower}, with $\sss$ the {\em stem} and $R$ the {\em blossom}.  The point $\sss \cap R$ is the {\em base} of the blossom, and the other end of $\sss$ is the {\em base} of both the stem and the flower.
\end{defin}

Now suppose $\sss'$ is a properly embedded arc in $A$ from the base of the blossom $R$ to a point $p$ in $\Sss - \sss$.  See Figure \ref{fig:swapRev} for an example when $p, q$ lie on edges of the spine.  

\begin{prop}[Stem Swapping]  \label{prop:stemswap}
The complex $\Sss'$ obtained from $\Sss$ by replacing the arc $\sss$ with the arc $\sss'$ is, up to isotopy, also a spine for $B$. That is, $T$ is isotopic in $M$ to the boundary of a regular neighborhood of $\Sss'$.
\end{prop}

    \begin{figure}[th]
\labellist
\small\hair 2pt
\pinlabel  $p$ at 85 20
\pinlabel  $\sss$ at 200 140
\pinlabel  $R$ at 260 120
\pinlabel  $q$ at 200 100
\pinlabel  $\sss'$ at 30 80
\endlabellist
    \centering
    \includegraphics[scale=0.5]{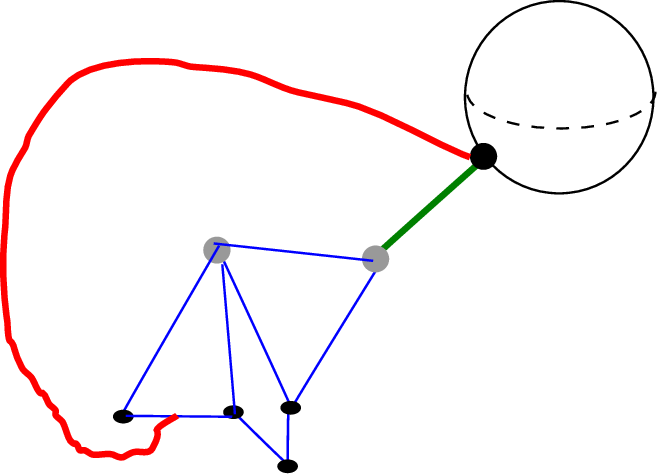}
\caption{A stem swap for the case $p, q \notin \bdd_- B \subset \Sss$ }
 \label{fig:swapRev}
    \end{figure}

\begin{proof}
Given the spine $\Sss$ as described, there is a natural alternate Heegaard splitting for $M$ in which $R$ is regarded as lying in $\bdd_- A$ instead of $\bdd_- B$.  It is obtained by deleting the flower $\sss \cup R$ from $\Sss$, leaving $R$ as an additional component of $\bdd_- A$.  Call the resulting spine $\Sss_-$ and let $A_+$ be the complementary compression body (so $M = A_+ \cup_{T'} \eta(\Sss_-)$).  Apply the argument of Lemma \ref{lemma:Zupan} to $A_+$, with $\beta = \sss, \aaa = \sss'$ and $r = R$.  (See Phase 1 of the proof of Proposition \ref{prop:snugdisjoint} for how we can regard the sphere $R$ as the point $r$.)
Let $\gamma$ be the path in $\bdd_+ A_+ = \bdd(\eta(\Sss_-))$ given by Lemma \ref{lemma:Zupan}.  Note that in Figure \ref{fig:swapRev} some edges in the spine $\Sss_-$ are shown, but we do not claim that the path $\gamma$ from Lemma \ref{lemma:Zupan} is a subgraph of $\Sss_-$.  Rather, the path is on the boundary of a {\em regular neighborhood} of $\Sss_-$ and does not necessarily project to an embedded path in $\Sss_-$ itself.  Note further that after the stem swap the edge in $\Sss$ that contains $p$ in its interior (if $p$ is on an edge and not on $\bdd_- B$) becomes two edges in $\Sss'$ and, dually, when $q$ is not on $\bdd_- B \subset \Sss$, it is natural to concatenate the two edges of $\Sss$ that are incident to $q$ into a single edge of $\Sss'$.

Returning to the original splitting, sliding an end of $\sss$ along $\gamma$ does not change the fact that $\Sss$ is a spine for $B$ and, viewing $T$ as the boundary of a regular neighborhood of $\Sss$, the slide defines an isotopy of $T$ in $M$.  After the slide, according to  Lemma \ref{lemma:Zupan}, $\sss$ and $\sss'$ have the same endpoints at $R$ and $p$; then $\sss$ can be isotoped to $\sss'$ rel its end points , completing the proof.  (Note that passing $\sss$ through $\sss'$, as must be allowed to invoke Lemma \ref{lemma:Zupan}, has no significance in this context.)
\end{proof}

\begin{defin} \label{defin:stemswap} The operation of Proposition \ref{prop:stemswap} in which we replace the stem $\sss$ with $\sss'$  is called a {\em stem swap}.  If the base of the stem $\sss'$ is the same as that of $\sss$ it is called a {\em local stem swap}.
\end{defin}

\begin{defin} \label{defin:edgereduce} Suppose $M = A \cup_T B$, and $\Sss$ is a spine for $B$.  A sphere $R_e$ that intersects $\Sss$ in a single point in the interior of an edge $e$ is an {\em edge-reducing sphere} for $\Sss$ 
and the associated edge $e$ is called a {\em reducing edge} in $\Sss$. 
\end{defin}

There is a broader context in which we will consider stem swaps:  Let $\Rrr$ be an embedded collection of edge-reducing spheres for $\Sss$, chosen so that no edge of $\Sss$ intersects more than one sphere in $\Rrr$.   (The latter condition, that each edge of $\Sss$ intersect at most one sphere in $\Rrr$, is discussed at the beginning of section \ref{sect:reduce}.) Let $\Mrr$ be a component of $M - \Rrr$ and $\Rrr_0 \subset \Rrr$ be the collection incident to $\Mrr$.  (Note that a non-separating sphere in $\Rrr$ may be incident to $\Mrr$ on both its sides.  We will be working with each side independently, so this makes very little difference in the argument.)
 
For a sphere $R_e \in \Rrr_0$, and $e \in \Sss$ the corresponding edge, the segment (or segments) $e \cap \Mrr$ can each be regarded as a stem in $\Mrr$, with blossom (one side of) $R_e$.  A stem swap on this flower can be defined for an arc $\sss' \subset \Mrr $ with interior disjoint from $\Sss$ that runs from the point $e \cap R_e$ to a point in $\Sss \cap \Mrr$.  Such a swap can be viewed in $M$ as a way of replacing $e$ with another reducing edge $e'$ for $R_e$ that differs from $e$ inside of $\Mrr$, leaving the other segment (if any) of $e$ inside $\Mrr$ alone.  
 
\begin{cor}\label{cor:stemswap} 
  If $\sss$ and $\sss'$ both lie in $\Mrr$, then the isotopy of $T$ described in Proposition \ref{prop:stemswap} can be assumed to take
place entirely in $\Mrr$.
\end{cor}

\begin{proof}  The manifold $\Mrr$ has a natural Heegaard splitting $\Mrr = A_{\Rrr} \cup_{T_0} B_{\Rrr}$ induced by that of $M$, in which each boundary sphere $R \in \Rrr_0$ is assigned to $\bdd_-  B_{\Rrr}$.  We describe this construction:  

Recall the setting: $\Sss$ is a spine for $B$ and $B$ itself is {\em a thin regular neighborhood} of $\Sss$.  Thus an edge-reducing sphere $R \in \Rrr$ intersects $B$ in a tiny disk, centered at the point $R \cap \Sss$.  This disk is a meridian of the tubular neighborhood of the reducing edge that contains the point $R \cap \Sss$.  The rest of $R$, all but this tiny disk, is a disk lying in $A$.  So $R$ is a reducing sphere for the Heegaard splitting of $M$.  

In the classical theory of Heegaard splittings (see e. g. [Sc]) such a reducing sphere naturally induces a Heegaard splitting for the manifold $\overline{M}$ obtained by reducing $M$ along $R$; that is, $\overline{M}$ is obtained by removing an open collar $\eta(R)$ of the sphere $R$ and {\em attaching $3$-balls} to the two copies $R_{\pm}$ of $R$ at the ends of the collar.  The classical argument then gives a natural Heegaard splitting on each component of $\overline{M}$: replace the annulus $T \cap \eta(R)$ by equatorial disks in the two balls attached to $R_{\pm}$.  Translated to our setting, the original spine $\Sss$ thereby induces a natural spine on each component of $\overline{M}$: the reducing edge is broken in two when $\eta(R)$ is removed, and at each side of the break, a valence-one vertex is attached, corresponding to the attached ball. 

For understanding $\Mrr$, we don't care about $\overline{M}$ and the unconventional (because of the valence one vertex) spine just described.  We care about the manifold $M - \eta(R)$, in which there are two new sphere boundary components created, but no balls are attached.  But the classical construction suggests how to construct a natural Heegaard splitting for the manifold $M - \eta(R)$ and a natural spine for it: simply regard both spheres $R_{\pm}$ as new components of $\bdd_- B$ and attach them at the breaks in the reducing edge where, above, we had added a valence 1 vertex.  This Heegaard splitting for $M - \eta(R)$ is topologically equivalent to taking the classical construction of the splitting on $\overline{M}$
 and removing two balls from the compression body $\overline{B}$.  

When applied to all spheres in $\Rrr$ simultaneously, the result of this construction is a natural Heegaard splitting on each component of  $M - \eta(\Rrr)$.  On $\Mrr$ it gives the splitting  $A_{\Rrr} \cup_{T_0} B_{\Rrr}$ which was promised above, and also a natural spine $\Sss_{\Rrr}$ for $B_{\Rrr}$.  
The required isotopy then follows, by applying Proposition \ref{prop:stemswap} to the Heegaard splitting $\Mrr = A_{\Rrr} \cup_{T_0} B_{\Rrr}$, with $B_{\Rrr}$ a thin regular neighborhood of the spine $\Sss_{\Rrr}$. 
\end{proof}

Suppose, in a stem swap, that $\sss'$ intersects an edge-reducing sphere $R_f$, with associated edge $f \neq \sss$.  See the first panel of Figure \ref{fig:bead3}.  (Note that $f$ is an edge in $\Sss$ but if $p \in f$ then $f$ becomes two edges in $\Sss'$.)  Although $R_f$ is no longer an edge-reducing sphere for $\Sss'$, there is a natural way to construct a corresponding edge-reducing sphere $R'_f$ for $\Sss'$, one that intersects $f$ in the same point, but now intersects $\sss$ instead of $\sss'$.  At the closest point in which $\sss'$ intersects $R_f$, tube a tiny neighborhood  in $R_f$ of the intersection point to its end at $R$ and then around $R$.  Repeat until the resulting sphere is disjoint from $\sss'$, as shown in the second panel of Figure \ref{fig:bead3}. 
One way to visualize the process is to imagine ambiently isotoping $R'_f$, in a neighborhood of $\sss'$,  to the position of $R_f$, as shown in the third panel of Figure \ref{fig:bead3}.  The effect of the ambient isotopy is as if $R$ is a bead sitting on the imbedded arc $\sss \cup \sss'$ and the ambient isotopy moves the bead along this arc and through $R_f$.  We will call $R'_f$ the {\em swap-mate} of $R_f$ (and vice versa).  

    \begin{figure}
    \labellist
\small\hair 2pt
\pinlabel  $R_f$ at 105 185
\pinlabel  $f$ at 85 120
\pinlabel  $R$ at 20 170
\pinlabel  $\sss'$ at 140 120
\pinlabel  $p$ at 130 60
\pinlabel  $q$ at 85 20
\pinlabel  $\sss$ at 50 100
\pinlabel $R'_f$ at 330 180
\pinlabel $R'_f$ at 480 180
\pinlabel  $R$ at 530 150
\endlabellist
    \centering
    \includegraphics[scale=0.6]{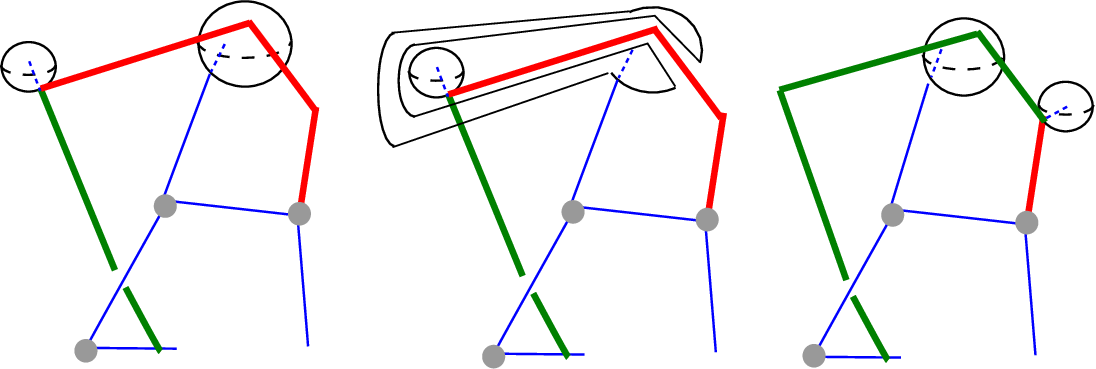}
\caption{Blossoms $R_f$ and $R'_f$} 
 \label{fig:bead3}
    \end{figure}

Here is an application.

Suppose $R_0$ is a reducing sphere for a reducing edge $e_0 \in \Sss$ and $\sss \subset e_0$ is one of the two segments into which $R_0$ divides $e_0$.  Let $\sss' \subset A - R_0$ be an arc whose ends are the same as those of $\sss$ but is otherwise disjoint from $\sss$.  Let $e'_0$ be the arc obtained from $e$ by replacing $\sss$ with $\sss'$.  Let $\eta(R_0)$ be the interior of a collar neighborhood of $R_0$ on the side away from $\sss$.

Viewing $\sss \cup R_0$ as a flower in the manifold $M - \eta(R_0)$, and the substitution of $\sss'$ for $\sss$ as a local stem swap, it follows from the proof of Proposition \ref{prop:stemswap} that the $1$-complex $\Sss'$ obtained from $\Sss$ by replacing $e_0$ with $e'_0$ is also a spine for $B$.  That is, $T$ is isotopic in $M$ to the boundary of a regular neighborhood of $\Sss'$.  Moreover, $e'_0$ remains a reducing edge in $\Sss'$ with edge-reducing sphere $R_0$. 

With this as context, we have:

\begin{lemma}  \label{lemma:simpleswap}  Suppose $\calE$ is a collection of edges in $\Sss$, with $e_0 \in \calE$, and let $\calE_r \subset \calE$ be the set of reducing edges for $\Sss$ that lie in $\calE$.  Similarly, suppose $\calE'$ is a collection of edges in $\Sss'$ containing the edge $e'_0$ constructed above, and $\calE'_r \subset \calE'$ is the set of reducing edges for $\Sss'$ that lie in $\calE'$.  If $\calE' - e'_0 \subset \calE - e_0$ then $\calE'_r - e'_0 \subset \calE_r - e_0$.
\end{lemma}

\begin{proof} Let $f$ be an edge in $\calE'_r$ other than $e'_0$, and $R'_f$ be a corresponding edge-reducing sphere for $\Sss'$.  Then $R'_f$ is disjoint from $e'_0$, so, although it may intersect $e_0$, any intersection points lie in $\sss \subset e_0$.  The swap-mate $R_f$ of $R'_f$ then may intersect $\sss'$ but by construction it will not intersect $\sss$.  Hence $R_f$ is disjoint from $e_0$ (as well as all edges of $\Sss$ other than $f$).  Hence $R_f$ is an edge-reducing sphere for $\Sss$ and $f \in \calE_r$.  
\end{proof}

Consider as usual a Heegaard splitting $M = A \cup_T B$, where $B$ is viewed as a thin regular neighborhood of a spine $\Sss$.  Suppose $\calE$ is a collection of edges in $\Sss$ and $\calE_r \subset \calE$ is the set of reducing edges for $\Sss$ that lie in $\calE$. (For example, $\calE$ might be the set of edges that intersects a specific essential sphere $S$ in $M$, as in the discussion that will follow Corollary \ref{cor:stemswap}.  This motivates the appearance of the red parallelograms in Figure \ref{fig:Rrimpact}.)  Suppose $\Rrr$ is an embedded collection of edge-reducing spheres for $\Sss$, one associated to each edge in $\calE_r$.  Let $\Mrr$ be a component of $M - \Rrr$ and consider a sphere $R_0 \in \Rrr_0 \subset \bdd \Mrr$.  Then, as just described before Lemma \ref{lemma:simpleswap}, a segment of the associated reducing edge $e_0$ that lies in $\Mrr$ can be regarded in $\Mrr$ as a stem $\sss$ with blossom $R_0$. (The rest of $e_0$ is shown as a dotted extension in Figure \ref{fig:Rrimpact}.) Let $\sss'$ be another arc properly embedded in $\Mrr$ which has the same ends as $\sss$ but is otherwise disjoint from $\Sss$, and let $\Sss'$ be the spine for $B$ constructed as above for the local stem swap of $\sss$ to $\sss'$.  Notice that because $\inter(\Mrr)$ is disjoint from the spheres $\Rrr$, $\inter(\sss') \subset \inter(\Mrr)$ is also disjoint from $\Rrr$.  

\begin{prop}  \label{prop:simpleswap}  Suppose $\calE'$ is a subcollection of the edges $\calE - e_0$, together possibly with the edge $e'_0$, and denote by $\calE'_r \subset \calE'$ the set of reducing edges for $\Sss'$ in $\calE'$.  
There is a collection of edge-reducing spheres $\Rrr'$ for $\Sss'$, one associated to each edge in $\calE'_r$, so that $\Rrr' \subset \Rrr$.
\end{prop}

\begin{proof} From Lemma \ref{lemma:simpleswap} we know that $\calE'_r - e'_0 \subset \calE_r - e_0$,
Since $\sss'$ is in $\Mrr$, it is disjoint from $\Rrr$, so for each edge $f$ in $\calE'_r - e'_0$ we can just use the corresponding edge reducing sphere for $f$ in $\Sss$.  In the same vein, since $R_0$ is disjoint from $\sss'$, $R_0$ is an edge-reducing sphere for  $e'_0$ in $\Sss'$.  
\end{proof}

There is an analogous result for more general stem swaps, but it is more difficult to formulate and prove.  To that end, suppose $\sss' \subset \Mrr$ has one end at the base of $R_0$ and the other at a point $p \in \Sss$.  Here $p$ is not a vertex of $\Sss$, nor a point in $\Rrr$ and $\inter(\sss')$ is disjoint from $\Sss$.  If $p$ lies on an edge of $\Sss$, the edge is not one that is also incident to the base point $q$ of $\sss$. 

Consider the stem swap as described in Proposition \ref{prop:stemswap}. 
After the stem swap, one difference between the two spines $\Sss$ and $\Sss'$ (other than the obvious switch from $\sss$ to $\sss'$) is that if $p$ lies on an edge $e \subset \Sss$ then $e$ becomes two edges $e_{\pm}$ in $\Sss'$ and if the base point $q$ of $\sss$ lies on an edge $e' \subset \Sss'$ then $e'$ began as two edges $e'_{\pm}$ in $\Sss$. See Figure \ref{fig:Rrimpact}.

   \begin{figure}
\labellist
\small\hair 2pt
\pinlabel  $e'$ at 380 60
\pinlabel  $e'_+$ at 80 85
\pinlabel  $e'_-$ at 95 45
\pinlabel  $S$ at 20 80
\pinlabel  $S$ at 128 72
\pinlabel  $S$ at 325 100
\pinlabel  $\sss$ at 130 90
\pinlabel  $p$ at 305 30
\pinlabel  $e_+$ at 285 12
\pinlabel  $e_-$ at 320 12
\pinlabel  $q$ at 105 65
\pinlabel  $e$ at 40 15
\pinlabel  $\sss'$ at 240 50
\endlabellist
    \centering
    \includegraphics[scale=0.8]{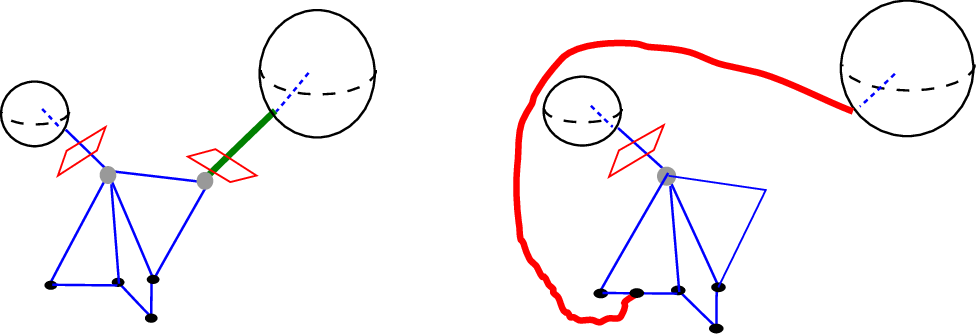}
\caption{Spines $\Sss$ and $\Sss'$} 
 \label{fig:Rrimpact}
    \end{figure}

\begin{defin} \label{defin:swapconsistent}
A collection of edges $\calE'$ in $\Sss'$ is {\em consistent with the swap} of $\sss$ to $\sss'$ (or  {\em swap-consistent}) if, when $p$ and/or $q$ lie on edges as just described, $\calE'$ has these properties:
\begin{itemize}
\item $\calE' - \{e_{\pm}, e', \sss'\} \subset \calE$.  
\item If either $e_{\pm}$ is in $\calE'$ then $e \in \calE$.
\item If both $e'_{\pm} \notin \calE'$ then $e' \notin \calE'$.  Or, equivalently, if $e' \in \calE'$ then at least one of $e'_{\pm} \in \calE$. 
\item Suppose $e$ is a reducing edge in $\calE$ with $R_e$ the corresponding edge-reducing sphere in $\Rrr$. Then the segment $e_+$ or $e_-$ not incident to $R_e$ is not in $\calE'$.  There must be such a segment since by hypothesis $p \notin \Rrr$.
\end{itemize}
(In the case that $p$ and/or $q$ lie on $\bdd_-B \subset \Sss$, so the edges $e$ and/or $e'$ are not defined, statements about these edges are deleted.)
\end{defin}

\begin{lemma}  \label{lemma:stemswap1}  Suppose $\calE'$ is consistent with the swap described above.  Then there is collection of edge-reducing spheres $\Rrr'$ for $\Sss'$, one associated to each reducing edge in $\calE'$, so that $\Rrr' \subset \Rrr$.
\end{lemma}


\begin{proof}  Consider any reducing edge $f \in \calE'$. If $f = \sss'$ use $R_0$ for the corresponding sphere in $\Rrr'$.  In any other case, since $f$ is a reducing edge for an edge in $\Sss'$, a corresponding edge-reducing sphere $R'_f$ is automatically disjoint from $int(\sss')$ since $R'_f$ only intersects $\Sss'$ in a single point.  Its swap-mate $R_f$ is then an edge-reducing sphere for $\Sss$, because it is disjoint from $int(\sss)$. We do not know that $R_f \in \Rrr$ and in fact it can't be if $int(\sss')$ intersects $R_f$, since $\sss'$ was chosen, following Proposition \ref{prop:simpleswap},  to be in $\Mrr$. 
With this in mind, consider the possibilities:

If $f \notin \{e_{\pm}, e', \sss'\}$ then $f \in \calE$, since $\calE'$ is consistent with the swap.  Then $R_f$ is an edge-reducing sphere for $f$ in $\Sss$, so $f$ is a reducing edge in $\calE$.  As originally defined prior to Proposition \ref{prop:simpleswap}, $\calE_r$ is the set of reducing edges in $\calE$, so $f \in \calE_r$. Since $\Rrr$ contains an edge-reducing sphere for each edge in $\calE_r$, $\Rrr$ contains an edge-reducing sphere for $f$.  By construction this sphere is disjoint from both $int(\sss)$ and $int(\sss')$, the latter by choice of $\sss'$. Include this as the sphere in $\Rrr'$ that corresponds to $f$.

As noted at the start, if $f = \sss'$ use $R_0$.  

If $f = e'$, then one of $e'_{\pm}$, say $e'_+$ is in $\calE$, since $\calE'$ is consistent with the swap. $R'_f$ may as well be taken to pass through $e'_+ \subset e'$.  Then $R_f$ is an edge-reducing sphere for $\Sss$ that passes through $e'_+$.  Hence $e'_+$ is a reducing edge in $\calE$.  The edge-reducing sphere in $\Rrr$ corresponding to $e'_+$ is again disjoint from both $int(\sss)$ and $int(\sss')$. Include this as the sphere in $\Rrr'$ that corresponds to $f$.  

If $f$ is one of the edges $e_{\pm}$, say $e_+$, then $e \in \calE$, since $\calE'$ is consistent with the swap.  As before, the sphere $R_f$ shows that $e$ is a reducing edge for $\Sss$ and so has a corresponding edge-reducing sphere $R$ in $\Rrr$.  Include it in $\Rrr'$ to correspond to $f = e_+$.  The last condition in Definition \ref{defin:swapconsistent} ensures that $e_- \notin \calE'$ so no corresponding edge-reducing sphere is included in $\Rrr'$.  In simple terms, $R$ appears only once in $\Rrr'$.  The condition also ensures that $f$ is the subedge of $e$ in $\Sss'$ that is incident to $R$.  
\end{proof}

\section{When $\bdd S \subset \bdd_- B \subset \bdd M$: early considerations.} \label{sect:early}

We will begin the proof of Theorem \ref{thm:main} in the case that $S$ is connected.  In conjunction with Proposition \ref{prop:S0} that will complete the proof of Theorem \ref{thm:main}.  

\subsection{Preliminary remarks} What will be most important for our purposes is not that $S$ is connected, but that $S$ is entirely disjoint either from all of $\bdd_- A$ or all of $\bdd_- B$, as is naturally the case when $S$ is connected.  So we henceforth assume with no loss of generality that $\bdd S \subset \bdd_- B$.  Following that assumption, the compression bodies $A$ and $B$ play very different roles in the proof.  We will be studying spines of $B$ and will take for $A$ the complement in $M$ of a regular neighborhood $\eta(\Sss)$ of such a spine $\Sss$.  In particular, each sphere component $R$ of $\bdd_- B$ is part of $\Sss$.  As noted in the discussion of spines following Definition \ref{defin:comcoll}, we can choose $\Sss$ so that each sphere component $R$ is incident to exactly one edge of $\Sss$; in that case we are in a position to apply the key idea of stem swapping to alter $\Sss$, as in Proposition \ref{prop:stemswap}.  

In contrast, the sphere components of $\bdd_- A$ play almost no role in the proof, other than requiring a small change in language.  Since in Theorem \ref{thm:main}  the isotopy class of $S$ remains fixed (indeed, that is the point of the theorem), we must be careful not to pass any part of $S$ through a sphere component of $\bdd_- A$, but the constructions we make use of will avoid this.  For example, underlying a stem swap in $\Sss$ is the slide and isotopy of an edge of $\Sss$. (See Proposition \ref{prop:stemswap}.) But these can be made to avoid sphere components of $\bdd_- A$, essentially by general position.  More explicitly, let $\hat{M}$ be the $3$-manifold obtained from $M$ by attaching a ball to each sphere component of $\bdd_- A$.  A slide or isotopy of an edge of $\Sss$ can avoid the centers of these balls by general position, and then be radially moved outside the entire balls and back into $A$.

A more subtle problem arises when, for example, we want to use a classical innermost disk (or outermost arc) argument to move a surface $F$ in $A$ so that it is disjoint from $S$.  In the classical setting we find a circle $c$ in $F \cap S$ that bounds a disk $E_S \subset S - F$ and a disk $E_F \subset F$ and argue that one can isotope $E_F$ past $E_S$, reducing the number of intersections, via a ball whose boundary is the sphere $E_F \cup E_S$.  But the existence of such a ball requires $A$ to be irreducible, an assumption that fails when $\bdd_- A$ contains spheres.  It will turn out that this fraught situation can always be avoided here by {\em redefining} $F$ to be the surface obtained by a simple disk-exchange, replacing $E_F \subset F$ with a push-off of $E_S \subset S$. 

A useful way to visualize and describe this process of redefining $F$ is to imagine, both in the argument and in the figures, a host of bubbles floating around in $A$, corresponding to sphere components of $\bdd_- A$.  These bubbles cannot pass through $S$ (or $\Sss$), but typically each bubble can pass ``through'' other surfaces we construct, in the sense that, when needed, the constructed surface $F$ can be redefined to pass on the other side of the bubble.  As shorthand for this process (which we have already seen in Phase 2 of the proof of Proposition \ref{prop:snugdisjoint}) we will describe the process as a {\em porous isotopy} of $F$ (equivalent to an actual isotopy in $\hat{M}$), since the bubbles appear to pass through $F$.

\bigskip

\subsection{The argument begins:}  Let $\Sss$ denote a spine of $B$ and, as usual, take $B$ to be a thin regular neighborhood of $\Sss$.  

Let $(\Delta, \bdd \Delta) \subset (A, T)$ be a collection of meridian disks for $A$ that constitute a complete collection of meridian disks for $\hat{A}$, the compression body obtained from $A$ by capping off all spherical boundary components by balls.  Let $B_+ = B \cup \eta(\Delta)$; since $\Delta$ is complete for $\hat{A}$ the complement of $B_+$ is the union of punctured balls and a punctured collar of $\bdd_- A \subset \bdd M$.
The deformation retraction of $B$ to $\Sss$ will carry $\Delta$ to disks in $M - \Sss$; continue to denote these $\Delta$. 

Suppose an edge $e$ of $\Sss$ is disjoint from $\Delta$.  A point on $e$ corresponds to a meridian of $B$ whose boundary lies on $\bdd B_+$.  If it is inessential in $\bdd B_+$ then it bounds a disk in $A$, so such a meridian can be completed to a sphere intersecting $e$ in a single point.  In other words, $e$ is a reducing edge of $\Sss$.  

The other possibility is that the boundary of the meridian disk for $e$ is essential on $\bdd B_+$, so it, together with an essential curve in $\bdd_- A$ bounds an essential spanning annulus $a_e \subset A$.  Together, the meridian disk of $e$ and the annulus $a_e$ comprise a boundary reducing disk for $M$, in fact one that also $\bdd$-reduces the splitting surface $T$.  (In particular, the disk is aligned with $T$.)   We will eliminate from consideration this possibility by a straightforward trick, which we now describe.

\begin{lemma} \label{lemma:annuliindelta} There is a collection $\mathcal{C} \subset \bdd_-A$ of disjoint essential simple closed curves with the property that $\mathcal{C}$ intersects any essential simple closed curve in $\bdd_-A$ that bounds a disk in $M$.  
\end{lemma}

\begin{proof} Suppose $A_0$ is a genus $g \geq 1$ component of $\bdd_- A$.  By standard duality arguments, the collection $K \subset A_0$ of simple closed curves that compress in $M$ can generate at most a $g$-dimensional subspace of $H_1(A_0, \real) \cong \real^{2g}$.  More specifically, one can find a non-separating collection $c_1,...c_g$ of disjoint simple closed curves in $A_0$ so that $\mathcal{C}_- = \cup_{i=1}^g c_i$ generates a complementary $g$-dimensional subspace of  $H_1(A_0, \real)$, and therefore essentially intersects any {\em non-separating} curve in $K$. It is easy to add to $\mathcal{C}_-$ a further disjoint collection of $2g-3$ simple closed curves, each non-separating, so that the result $\mathcal{C}_0 \subset A_0$ has complement a collection of $2g-2$ pairs of pants.  Any curve in $A_0$ that is disjoint from $\mathcal{C}_0$ is parallel to a curve in $\mathcal{C}_0$ and so must be non-separating.  Since it is disjoint from $\mathcal{C}_- \subset \mathcal{C}_0$ it cannot be in $K$.


Do the same in each component of $\bdd_-A$; the result is the required collection $C$.  
\end{proof}

Following Lemma \ref{lemma:annuliindelta} add to the collection of disks $\Delta$ the disjoint collection of annuli $\mathcal{C} \times I \subset \bdd_-A \times I \subset M - B_+$, and continue to call the complete collection of meridional disks and these spanning annuli $\Delta$.   Then a meridian of an edge $e$ of $\Sss$ that is disjoint from the (newly augmented) $\Delta$ cannot be part of a $\bdd$-reducing disk for $T$ and so must be part of a reducing sphere.  Since the collection $S$ of reducing spheres and $\bdd$-reducing disks we are considering have no contact with $\bdd_-A$, arcs of $S \cap \Ddd$ are nowhere incident to $\bdd_- A$.  Additionally, no circle in $S \cap \Ddd$ can be essential in an annulus in $\mathcal{C} \times I $, since no circle in $\mathcal{C}$ bounds a disk in $M$.   Hence the annuli which we have added to $\Delta$ intersect $S$ much as a disk would: each circle of intersection bounds a disk in the annulus and each arc of intersection cuts off a disk from the same end of the annulus.  As a result, the arguments cited below, usually applied to disk components of $\Delta$, apply also to the newly added annuli components $\mathcal{C} \times I$.  

\section{Reducing edges and $S$}

\begin{lemma} \label{lemma:isolate} Suppose a spine $\Sss$ for $B$ and a collection $\Delta$ of meridians and annuli, as just described, have been chosen to minimize the pair $(|\Sss \cap S|, |\bdd \Delta \cap S|)$ (lexicographically ordered, with $\Sss, S, \Delta$ all in general position).  Then $\Sss$ intersects $int(S)$ only in reducing edges.
\end{lemma}

Notes: \begin{itemize}
\item We do not care about the number of circles in $\Ddd \cap S$.  
\item If $S$ is a disk and intersects $\Sss$ transversally only in $\bdd S \subset \bdd_-B$, then $S$ is aligned with $T = \bdd (\eta(\Sss))$ and intersects $B$ in a vertical annulus, completing the proof of Theorem \ref{thm:main} in this case.  In addition, $S$ is a $\bdd$-reducing disk for $T$ if $\bdd S$ is essential in $\bdd_- B$.  
\item If $S$ is a sphere and intersects $\Sss$ transversally only in a single point, then $S$ is aligned with $T$, completing the proof of Theorem \ref{thm:main} in this case.  Moreover, if the circle $S \cap T$ is essential in $T$, $S$ is a reducing sphere for $T$.
\end{itemize}
 
\begin{proof}
  
Recall from a standard proof of Haken's Theorem (see eg \cite{Sc1}, \cite[Proposition 2.2]{ST})  that $(\Sss \cup \Ddd) \cap S$ (ignoring circles of intersection) can be viewed as a graph $\Ggg$ in $S$ in which points of $\Sss \cap S$ are the vertices and $\Ddd \cap S$ are the edges.  As discussed in \cite{ST} in the preamble to Proposition 2.2 there, this is accomplished by extending the disks and annuli $\Ddd$ via a retraction $B \to \Sss$ 
so that it becomes a collection of disks and annuli whose imbedded interior is disjoint from $\Sss$ and whose singular boundary lies on $\Sss$. When $S$ is a disk we will, with slight abuse of notation, also regard $\bdd S$ as a vertex in the graph, since it lies in $\bdd_- B \subset \Sss$.  (This can be made sensible by imagining capping off $\bdd S$ by an imaginary disk outside of $M$.)  

Borrowing further from the preamble to \cite[Proposition 2.2]{ST}, an edge in $\Ggg$ is a loop if both ends lie on the same vertex, called the base vertex for the loop.  A loop is {\em inessential} if it bounds a disk in $S$ whose interior is disjoint from $\Sss$, otherwise it is {\em essential}.  A vertex in $\Ggg$ is {\em isolated} if it is incident to no edge in $\Ggg$.  

It is shown in \cite{ST} that if $\Sss$ and $\Delta$ are chosen to minimize  the pair $(|\Sss \cap S|, |\bdd \Delta \cap S|)$ then
\begin{itemize}
\item there are no inessential loops
\item any innermost loop in the graph $\Ggg$ bounds a disk in $S$ that contains only isolated vertices and
\item if there are no loops in $\Ggg$ then every vertex is isolated.
\end{itemize}

It follows that either $S$ is disjoint from $\Sss$ (so it is aligned and we are done) or there is at least one isolated vertex.  An isolated vertex represents a point $p$ in an edge $e$ of $\Sss$ which is incident to no element of $\Delta$.  The point $p$ defines a meridional disk $D_B$ of $B = \eta(\Sss)$, and the fact that the curve $\bdd D_B \subset \bdd_+ A$ is disjoint from $\Ddd$ ensures that $\bdd D_B$  is parallel to a curve in $\bdd_- A$ that is inessential.  Thus $\bdd D_B$ also bounds a disk $D_A$ in $A$.  Then $D_A \cup D_B$ is a reducing sphere, so $e$ is a reducing edge in $\Sss$. This establishes the original Haken's Theorem and, if there are no loops at all, also Lemma  \ref{lemma:isolate}.  That there are no loops is what we now show.

 Consider an innermost loop, consisting of a vertex $p \in \Sss \cap S$ and an edge lying in a component $D $ of $ \Ddd$.  Together, they define a circle $c$ in $S$ that bounds a disk $E \subset S$ whose interior, by the argument of \cite[Proposition 2.2]{ST}), contains only isolated vertices and so intersects $\Sss$ only in reducing edges.  Remembering that we are taking $A = M - \eta(\Sss)$, the $3$-manifold $A_- = A - \eta(D)$ can be viewed as $M - \eta(D \cup \Sss)$, so $c$ is parallel in $E$ to a circle $c'$ in $\bdd A_-$ bounding a subdisk $E_-$ of $E$.  $E_-$ is the complement in $E$ of the collar in $E$ between $c$ and $c'$. Since $E_-$ intersects $\Sss$ only in reducing edges, it follows immediately that $c'$ is null-homotopic in $A_-$ and then by Dehn's lemma that it bounds an embedded disk $E'$ entirely in $A_-$.  

By standard innermost disk arguments we can find such an $E'$ so that its interior is disjoint from $\Ddd$.  Now split $D$ in two by compressing the loop to the vertex along $E'$ and replace $D$ in $\Ddd$ by these two pieces, creating a new complete (for $\hat{A}$) collection of disks and annuli $\Ddd'$, with $|\bdd \Ddd' \cap S| \leq  |\bdd \Ddd \cap S| -2$.  Since we have introduced no new vertices, this contradicts our assumption that $(|\Sss \cap S|, |\bdd \Delta \cap S|)$ is minimal.
\end{proof}

Note that the new $\Ddd'$ may intersect $S$ in many more circles than $\Ddd$ did, but we don't care.

\section{Edge-reducing spheres for $\Sss$} \label{sect:reduce}

Recall from Section \ref{sect:stemswap} that, given a reducing edge $e$ in $\Sss$ an associated edge-reducing sphere $R_e$ is a sphere in $M$ that passes once through $e$.  
Any other edge-reducing sphere $R'_e$ passing once through $e$ is porously isotopic to $R_e$ in $M$ (i. e. isotopic in $\hat{M}$) via edge-reducing spheres.  Indeed, the segment of $e$ between the points of intersection with $\Sss$ provides an isotopy from the meridian disk $R_e \cap B$ to $R'_e \cap B$; this can be extended to a porous isotopy of $R_e \cap A$ to $R'_e \cap A$ since $\hat{A}$ is irreducible.  So $R_e$ is well-defined up to porous isotopy.

Let $\Sss$ be a spine for $B$ in general position with respect to the disk/sphere $S$, and suppose $\calE$ is a collection of edges in $\Sss$.  Let $\Rrr$ be a corresponding embedded collection of edge-reducing spheres transverse to $S$, one for each reducing edge in $\calE$.  Let $|\Rrr \cap S|$ denote the number of components of intersection.

\begin{defin} \label{defin:weight}
 The {\em weight} $w(\Rrr)$ of $\Rrr$ is $|\Rrr \cap S|$.
Porously isotope $\Rrr$ via edge-reducing spheres so that its weight is minimized, and call the result $\Rrr(\calE)$.  Then the {\em weight} $w(\calE)$ of $\calE$ is $w(\Rrr(\calE))$.  
\end{defin}

Consider the stem swap as defined in Proposition \ref{prop:stemswap} and Corollary  \ref{cor:stemswap} and suppose $\calE'$ is a collection of edges in $\Sss$ that is swap-consistent with $\calE$.  

\begin{lemma} \label{lemma:stemswap2}
There is a collection $\Rrr'$ of edge-reducing spheres for $\Sss'$, one for each reducing edge in $\calE'$ so that $w(\Rrr') \leq w(\Rrr)$.
\end{lemma}

\begin{proof}
This is immediate from Lemma \ref{lemma:stemswap1}.
\end{proof}

\begin{cor}  \label{cor:stemswap3} Suppose in Lemma \ref{lemma:stemswap2} $\Rrr$ is $\Rrr(\calE)$.  Then  $w(\calE') \leq w(\calE)$.
\end{cor}

\begin{proof}  Let $\Rrr'$ be the collection of spheres given in Lemma \ref{lemma:stemswap2}.  By definition $w(\calE') \leq w(\Rrr')$ so, by Lemma \ref{lemma:stemswap2} $$w(\calE') \leq w(\Rrr') \leq w(\Rrr) = w(\Rrr(\calE)) = w(\calE).$$\end{proof}

Here is a motivating example:  For $\Sss$ a spine of $B$ in general position with respect to $S$, let $\calE$ be the set of edges that intersect $S$, with the set of edge-reducing spheres $\Rrr = \Rrr(\calE)$ corresponding to the reducing edges of $\calE$ .  As usual, let $\Mrr$ be a component of $M - \Rrr$ and $\Rrr_0$ be the collection of spheres in $\bdd \Mrr$ that comes from $\Rrr$.  Suppose  $R_0$ is a sphere in $\Rrr_0$ with stem $\sss$, and suppose $\sss'$ is an arc in $\Mrr$ from the base of $R_0$ to a point $p$ in an edge $e$ of $\Sss$, very near an end vertex of $e$, so that the subinterval of $e$ between $p$ and the end vertex does not intersect $S$.

Perform an edge swap and choose $\calE'$ to be the set of edges in $\Sss'$ that intersect $S$.

\begin{prop} \label{prop:stemswap4}
$\calE'$ is swap-consistent with $\calE$.
\end{prop}

\begin{proof}  All but the last property of Definition \ref{defin:swapconsistent} is immediate, because $S$ will intersect an edge if and only if it intersects some subedge.  The last property of Definition \ref{defin:swapconsistent} follows from our construction:  Since $\sss'$ lies in a component $\Mrr$ of  $M - \Rrr$, the point $p$ lies between the sphere in $\Rrr$ corresponding to $e$ and an end vertex $v$ of $e$, and the segment of $e$ between $p$ and $v$ is disjoint from $S$ by construction and therefore not in $\calE'$.
\end{proof}

Define the weight $w(\Sss)$ of $\Sss$ to be $w(\calE)$, and similarly $w(\Sss') = w(\calE')$.   

\begin{cor} \label{cor:stemswap5}  Given a stem swap as described in Propositions  \ref{prop:stemswap} or \ref{prop:simpleswap}
 for $\Rrr(\calE)$, $w(\Sss') \leq w(\Sss)$.
\end{cor}

\begin{proof}  This follows immediately from Proposition \ref{prop:stemswap4} and Corollary \ref{cor:stemswap3}.
\end{proof} 

We will need a modest variant of Corollary \ref{cor:stemswap5} that is similar in proof but a bit more complicated.  
As before, let $\calE$ be the set of edges in a spine $\Sss$ that intersect $S$, with the set of edge-reducing spheres $\Rrr = \Rrr(\calE)$ corresponding to the reducing edges of $\calE$.   Suppose $e_0 \in \calE$ with corresponding edge-reducing sphere $R_0 \in \Rrr$. Then, by definition, $$w(\Sigma) = w(\calE) = w(\Rrr) = w(\Rrr - R_0) + w(R_0) = w(\Rrr - R_0) + |R_0 \cap S|.$$

Let $\Rrr_- = \Rrr - R_0$, $\calE_- = \calE - e_0$ and $\Mrr_-$ be the component of $M - \Rrr_-$ that contains $R_0$.  Perform an edge swap {\em in $\Mrr_-$} as in the motivating example: replace the stem $\sss$ of a sphere  $\fra$ in $\Rrr_-$ with $\sss'$, an arc in $\Mrr_-$ from the base of $\fra$ to a point $p$ in an edge $e$ of $\Sss$, very near an end vertex of $e$, so that the subinterval of $e$ between $p$ and the end vertex does not intersect $S$.   Notice that, in this set-up, $R_0$ is essentially invisible: the new stem $\sss'$ is allowed to pass through $R_0$.  The swap-mate $R'_0$  of $R_0$ is an edge-reducing sphere for $e_0$ in $\Sss'$ that is disjoint from $\Rrr_- = \Rrr - R_0$

As in the motivating example, let $\calE'$ be the set of edges in $\Sss'$ that intersects $S$ and further define $\calE'_- = \calE' - e_0$.  

\begin{prop} \label{prop:Rrr-}  
$w(\Sss') \leq w(\Sss) - |R_0 \cap S|+  |R'_0 \cap S|$

\end{prop} 

\begin{proof} As in the motivating example, $\calE'_-$ is consistent with the swap, so by Lemma \ref{lemma:stemswap1} there is a collection $\Rrr'_- \subset \Rrr_- = \Rrr - R_0$ of edge-reducing spheres associated to the edge-reducing spheres of $\calE'_-$. Then $\Rrr'_- \cup R'_0$ is a collection of edge-reducing spheres for $\calE'$.  Thus 
$$w(\Sss') = w(\calE') \leq w(\Rrr'_-) + w(R'_0) \leq w(\Rrr_-) + w(R'_0) =$$ $$ w(\Rrr) - w(R_0) + w(R'_0) = w(\Sss) - w(R_0) + w(R'_0)$$
\end{proof}

 \section{Minimizing $w(\Rrr) = |\Rrr \cap S|$}

Following Lemma \ref{lemma:isolate}, consider all spines that intersect $S$ only in reducing edges, and define $\calE$ for each such spine to be as in the motivating example from Section \ref{sect:reduce}: the collection of edges that intersect $S$. Let $\Sss$ be a spine for which $w(\Sss) = w(\calE)$ is minimized and let $\Rrr(\Sss)$ denote the corresponding collection of edge-reducing spheres for $\Sss$.  In other words, among all such spines and collections of edge-reducing spheres, choose that which minimizes the number $|\Rrr \cap S|$ of (circle) components of intersection. 

\begin{prop} \label{prop:main2}  $\Rrr(\Sss)$ is disjoint from $S$.  
\end{prop}

Note that for this proposition we don't care about how often the reducing edges of the spine $\Sss$ intersects $S$.  We revert to the notation $\Rrr$ for $\Rrr(\Sss)$.  

\begin{proof}  
Suppose, contrary to the conclusion, $\Rrr \cap S \neq \emptyset$.  Among the components of $\Rrr \cap S$ pick $c$ to be one that is innermost in $S$.   Let $E \subset S$ be the disk that $c$ bounds in $S$ and let $M_{\Rrr}$ be the component of $M - \Rrr$ in which $E$ lies.  
Let $R_0 \in \Rrr$ be the edge-reducing sphere on which $c$ lies, $e_0 \subset \Sss$ the corresponding edge, 
$p$ be the base $e_0 \cap R_0$ of $R_0$, and $D \subset R_0$ be the disk $c$ bounds in $R_0 - p$.  Finally, as in Proposition \ref{prop:Rrr-} let $\Rrr_- = \Rrr - R_0$ and $\Mrr_- \supset \Mrr$ be the component of $M - \Rrr_-$ that contains $R_0$.
\bigskip

{\bf Claim 1:}  After local stem swaps as in Proposition \ref{prop:simpleswap} we can take $e_0$ to be disjoint from $E$.

{\em Proof of Claim 1:} Let $v_{\pm}$ be the vertices at the ends of $e_0$, with $e_{\pm}$ the incident components of $e_0 - p$.  In a bicollar neighborhood of $R_0$, denote the side of $R_0$ incident to $e_{\pm}$ by respectively $R_{\pm}$, with the convention that a neighborhood of $\bdd E$ is incident to $R_+$.  It is straightforward to find a point $p' \in R_0$ and arcs $e'_{\pm}$ in $\Mrr - E$, each with one end at the respective vertex $v_{\pm}$ and other end incident to $p'$ via the respective side $R_{\pm}$.  

It is not quite correct that replacing each of $e_{\pm}$ with $e'_{\pm}$ is a local stem swap, since the arcs are incident to $R_0$ at different points.  But this can be easily fixed: Let $\gamma$ be an arc from $p'$ to $p$ in $R_0$ and $\gamma_{\pm}$ be slight push-offs into $R_{\pm}$.  Then replacing each $e_{\pm}$ with respectively $e'_{\pm} \cup \gamma_{\pm}$ is a local stem swap.  Attach the two arcs at $p \in R_0$ to get a new reducing edge $e'_0$ for $R_0$, and then use the arc $\gamma$ to isotope $e'_0$ back to the reducing edge $e'_+ \cup e'_-$, which is disjoint from $E$, as required.  See Figure \ref{fig:claim1swap}.   Revert to $e_0, p$ etc as notation for $e'_+ \cup e'_-$, now disjoint from $E$.  This proves Claim 1.

    \begin{figure}[th]
     \labellist
\small\hair 2pt
\pinlabel  $R_0$ at 40 70
\pinlabel  $p'$ at 20 85
\pinlabel  $v_+$ at 80 90
\pinlabel  $p$ at 40 175
\pinlabel  $v_-$ at 15 195
\pinlabel  $\gamma$ at 35 120
\pinlabel  $E\subset S$ at 100 200
\pinlabel  $e_+$ at 63 120
\pinlabel  $e'_+$ at 50 108
\pinlabel  $e'_-$ at 10 140
\pinlabel  $R_+$ at 80 250
\pinlabel  $R_-$ at 30 250
\pinlabel  $e'_+\cup\gamma_+$ at 230 115
\pinlabel  $e'_+\cup e'_-$ at 380 110
\endlabellist
    \centering
    \includegraphics[scale=0.8]{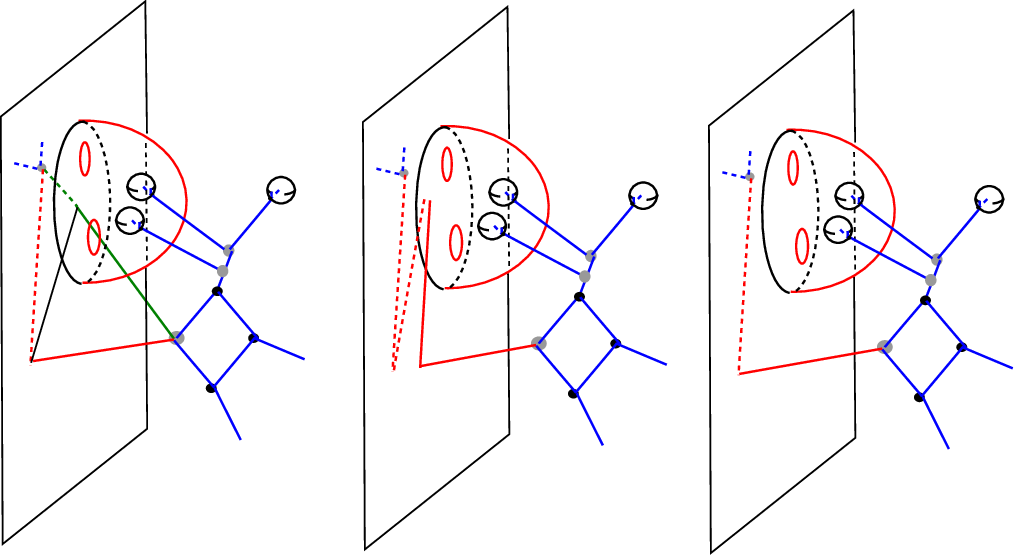}
\caption{Making $e_0$ disjoint from $E$ by local stem swaps} 
 \label{fig:claim1swap}
    \end{figure}
    \bigskip

{\bf Claim 2:}  After local stem swaps we can assume that each stem that intersects  $E$, intersects it always with the same orientation.

{\em Proof of Claim 2:} Figure \ref{fig:onept} shows how to use a local stem swap to cancel adjacent intersections with opposite orientations, proving the claim.
\bigskip

    \begin{figure}[th]
 \labellist
\small\hair 2pt
\pinlabel  $S$ at 75 105
\endlabellist
    \centering
    \includegraphics[scale=0.5]{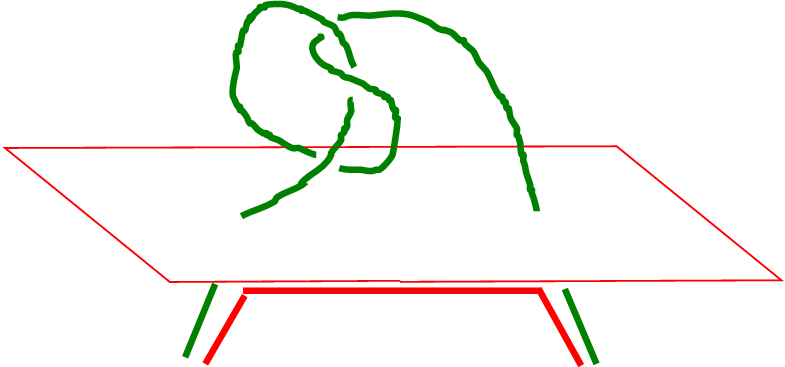}
\caption{A local stem swap} 
 \label{fig:onept}
    \end{figure}

Notice that if $E$ is non-separating in $\Mrr$ we could do a local stem swap so that each stem intersects $E$ algebraically zero times.  Following Claim 2, this implies that we could make all stems disjoint from $E$.  Once $E$ intersects no stems, replace the subdisk $D$ of $R_0$ that does not contain $p$ with a copy of $E$.  The result $R'_0$ is still an edge-reducing sphere for $e_0$, but the circle $c$ (and perhaps more circles) of intersection with $S$ has been removed.  That is 
$$w(R'_0) = |R'_0 \cap S| \leq |R_0 \cap S| -1 = w(R_0) - 1.$$
Hence $w(\Sss') < w(\Rrr) = w(\Sss)$, contradicting our hypothesis that $w(\Sss)$ is minimal.  

So we henceforth proceed under the assumption that $E$ is separating, but hoping for the same conclusion: that we can arrange for all stems to be disjoint from $E$, so that $R'_0$ as defined above leads to the same contradiction.  Since $E$ is separating, a stem that always passes through $E$ with the same orientation can pass through at most once.  So we henceforth assume that each stem that intersects $E$ intersects it exactly once.

In a bicollar neighborhood of the disk $E$, let $E_+$ be the side of $E$ on which $v_+$ lies, and $E_-$ be the other side of $E$.  Consider a stem $\sss$ of a boundary sphere $\fra$ of $\Mrr_-$.  If $\sss$ intersects $E$, the subsegment of $\sss - E$ that is incident to the blossom $\fra$ passes through one of $E_{\pm}$.  Let $\hat{\sss}_{\pm}$ be the collection of those stems intersecting $E$ for which this subsegment passes through respectively $E_{\pm}$.  If $\sss \in \hat{\sss}_+$, it is straightforward to find an alternate stem $\sss'$ from $\fra$ to a point very near $v_+$ so that $\sss'$ misses $E$.  A stem swap to $\sss'$ is as in Proposition \ref{prop:stemswap}, and so by Corollary \ref{cor:stemswap5} does not increase weight.  Hence we have proven:

{\bf Claim 3}:  After stem swaps, we may assume that each stem that intersects $E$ is in $\hat{\sss}_-$.
\bigskip

Following Claim 3, we move to swap those stems in $\hat{\sss}_-$ for ones that are disjoint from $E$.  Let $\sss$ be the stem of a boundary sphere $\fra$ of $\Mrr_-$, and assume that $\sss \in \hat{\sss}_-$.  Then it is straightforward to find an alternate stem $\sss'$ for $\fra$ that is disjoint from $E$ and ends in a point very near $v_-$, for example by concatenating an arc in $E_-$ with an arc in $R_-$ and an arc parallel to $e_-$.  See Figure \ref{fig:keyswap}.  A problem is, that such an arc intersects the disk $D \subset R_0$, so, after such a swap, $R_0$ is no longer an edge-reducing sphere for the new spine.  However, if such swaps are performed simultaneously on all stems in $\hat{\sss}_-$, we have seen that the swap-mate of $R_0$ is an edge-reducing sphere for the new spine $\Sss'$, as required.  But observe in Figure \ref{fig:keyswap} that the swap-mate is exactly $R'_0$!  So we can now appeal to Proposition \ref{prop:Rrr-}:
$$w(\Sss') \leq w(\Sss) - |R_0 \cap S|+  |R'_0 \cap S| \leq w(\Sss) - 1.$$
The contradiction proves Proposition \ref{prop:main2}.
\end{proof}

\begin{figure}[th]
\labellist
\small\hair 2pt
\pinlabel  $\sss'$ at 70 160
\pinlabel  $\sss$ at 120 190
\pinlabel  $v_+$ at 105 155
\pinlabel  $v_-$ at 15 70
\pinlabel $E$ at 120 270
\pinlabel $p$ at 40 90
\pinlabel  $D$ at 50 255
\pinlabel  $R_0$ at 80 340
\endlabellist
    \centering
    \includegraphics[scale=0.7]{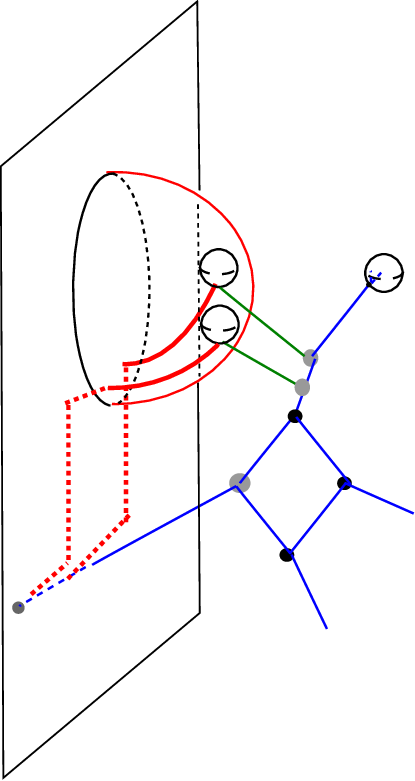}
\caption{} 
 \label{fig:keyswap}
    \end{figure}

\section{Conclusion} \label{sect:final}

\begin{prop}  \label{prop:final}
Suppose $\Sss$ intersects $S$ only in reducing edges, and the associated set $\Rrr$ of edge-reducing spheres is disjoint from $S$.  Then $T$ can be isotoped (via edge slides of $\Sss$) so that  $S$ is aligned with $T$.  \end{prop}

\begin{proof}  We will proceed by stem swaps, chosen so that they do not affect the hypothesis that $\Rrr \cap S = \emptyset$.  Let $\Mrr$ be the component of $M - \Rrr$ that contains $S$, and $\Rrr_0 \subset \bdd \Mrr$ the collection of sphere components that come from $\Rrr$. In $\Mrr$ each $\fra \in \Rrr_0$ is the blossom of a flower whose stem typically intersects $S$.  (A non-separating sphere in $\Rrr$ may appear twice in $\Rrr_0$, with one or both stems intersecting $S$.)  Denote by $\hat{\sss}$ the collection of all stems of $\Rrr_0$ that intersect $S$. The proof will be by induction on $|\hat{\sss} \cap S|$.  If $|\hat{\sss} \cap S| = 0$ then either $S$ is a sphere disjoint from $\Sss$ and therefore aligned, or $S$ is a disk.  In the latter case our convention of which compression body to call $B$ has $\bdd S \subset \bdd_- B \subset \Sss$, so $T \cap S$ is a single circle parallel to $\bdd S$ in $S$.  Again this means that $S$ is aligned.  Suppose then that $|\hat{\sss} \cap S| > 0$ and inductively assume that the Proposition is known to be true for lower values of $|\hat{\sss} \cap S|$.  Consider the possibilities:
\medskip

{\bf Case 1:} $S$ is a disk.

Since $|\hat{\sss} \cap S| > 0$ there is a blossom $\fra \in \Rrr_0$ with stem $\sss \in \hat{\sss}$.  Let $\sss_{\fra} \subset \sss$ be the segment of $\sss - S$ whose interior is disjoint from $S$ and whose end points are the blossom $\fra$ and a point $p_{\fra}$ in $S$. Let $\gamma$ be an arc in $S$ that runs from $p_{\fra}$ to $\bdd S$ that avoids all other points of $\hat{\sss} \cap S$.  Push the arc $\gamma \cup \sss_{\fra}$ off of $S$ in the direction of $\sss_{\fra}$ so that it becomes a stem $\sss'$ for $\fra$.  Do a stem swap from $\sss$ to $\sss'$, and let $\Sss'$ be the result.  See Figure \ref{fig:final2}. $\sss'$ is disjoint from $S$, so $\sss$ is thereby removed from $\hat{\sss}$, lowering  $|\hat{\sss} \cap S|$ by at least one.  The stem swap does not affect other reducing edges or their edge-reducing spheres, so the latter remain disjoint from $S$.  By Proposition \ref{prop:stemswap} $\Sss'$ is still a spine of $B$, so $T$ is isotopic in $M$ to a regular neighborhood of $\Sss'$.  The inductive hypothesis implies that then $T$ can be isotoped so that $S$ is aligned with $T$, as required.

 \begin{figure}
    \centering
    \labellist
\small\hair 2pt
 \pinlabel  $S$ at 150 110
  \pinlabel  $\sss_{\fra}$ at 45 160
  \pinlabel  $\fra$ at 95 190
\pinlabel  $\gamma$ at 30 120
\pinlabel  $p_{\fra}$ at 50 130
  \pinlabel  $\sss'$ at 300 130
\endlabellist
\centering
      \includegraphics[scale=0.7]{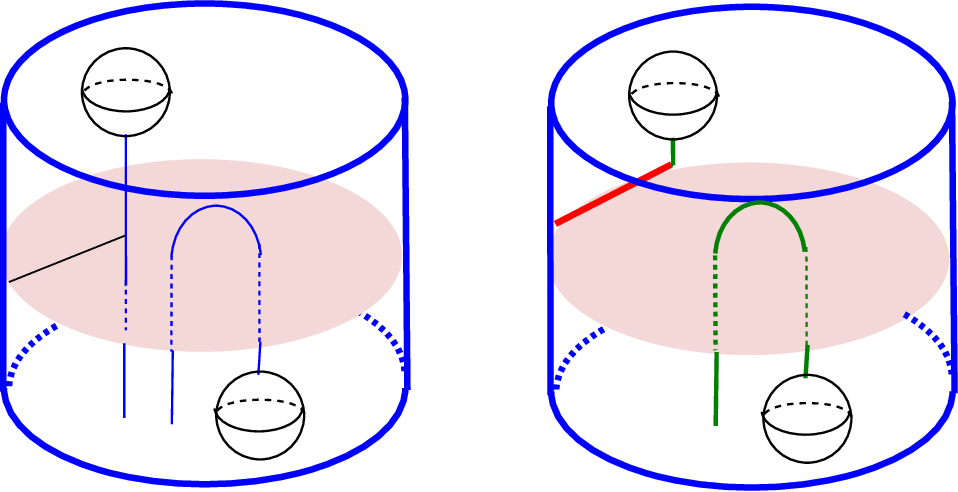}
\caption{Swap lowering $|\hat{\sss} \cap S|$, $S$ a disk} 
 \label{fig:final2}
    \end{figure}

\medskip

{\bf Case 2:} $S$ is a sphere.

Although $S$ could be non-separating in $M$, it cannot be non-separating in $\Mrr$.  Here is the argument: 
Suppose $S \subset \Mrr$ is non-separating.  If $\hat{\sss}$ were disjoint from $S$ then $S$ would have no intersections with the Heegaard surface $T$ at all and so $S \subset A$.  But in a compression body such as $A$, all spheres separate, a contradiction.  We will inductively reach the same contradiction by showing that if $\hat{\sss}$ does intersect $S$ there is a local stem swap that lowers $|\hat{\sss} \cap S|$: Since $S$ is non-separating there is a circle $c$ in $\Mrr - \Sss$ that intersects $S$ in a single point $p$.  Let $\gamma$ be a path in $S$ from $p$ to a point in $\sss \cap S$, where $\sss \in \hat{\sss}$ and $\gamma$ is chosen so that its interior is disjoint from $\hat{\sss}$.  Band sum $\sss$ to $\gamma$ along a band perpendicular to $S$, with $\gamma$ as its core.  The result is an edge $\sss'$ that is obtained from $\sss$ by a local stem swap snd intersects $S$ in one fewer point than $\sss$ does, as required.  See Figure \ref{fig:final3}.  

 \begin{figure}
    \centering
    \labellist
\small\hair 2pt
 \pinlabel  $S$ at 150 110
\pinlabel  $c$ at 15 150
\pinlabel  $\gamma$ at 40 100
\pinlabel  $p$ at 20 80
  \pinlabel  $\sss'$ at 280 140
\endlabellist
\centering
      \includegraphics[scale=0.7]{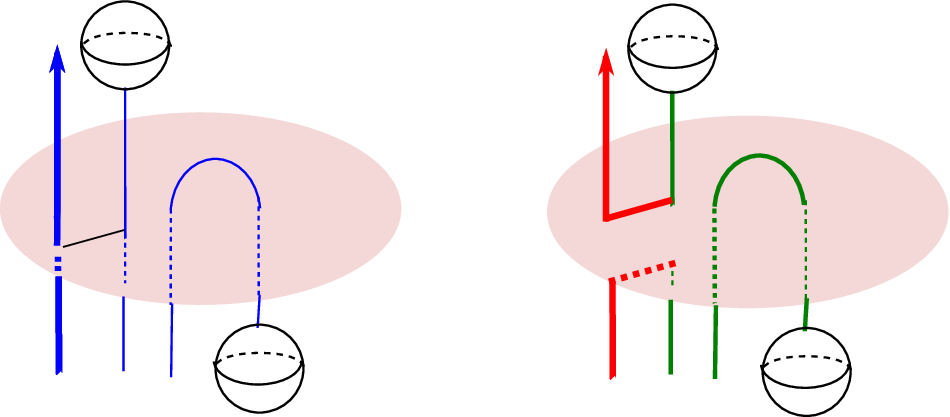}
\caption{Swap lowering $|\hat{\sss} \cap S|$, $S$ a non-separating sphere} 
 \label{fig:final3}
    \end{figure}

So $S$ is separating in $\Mrr$. This implies that no stem can intersect $S$ more than once algebraically and so, following local stem swaps as in Claim 2 of Proposition \ref{prop:main2} (see Figure \ref{fig:onept}), no  more than once geometrically.  If no stem intersects $S$ at all, then $S \subset A$ and so $S$ is aligned, finishing the proof.  

Suppose, on the other hand, there is at least one stem $\sss_!$ that intersects $S$ exactly once.  Repeat the argument of Case 1 for all stems other than $\sss_!$, using the point $p_! = \sss_i \cap S$ in place of $\bdd S$ in the argument. The result is that, after a sequence of stem swaps, all stems other than $\sss_!$ are disjoint from $S$.  This means that $S \cap \Sss$ consists of the single point $p_!$.  In other words, $T$ intersects $S$ in a single circle, and so $S$ is aligned.
 \end{proof}

The sequence of Proposition \ref{prop:S0}, Lemma \ref{lemma:isolate}, Proposition \ref{prop:main2}  and Proposition \ref{prop:final} establishes Theorem \ref{thm:main}.  \qed

\section{The Zupan example} \label{section:Zupan}

Some time ago, Alex Zupan proposed a simple example for which the Strong Haken Theorem seemed unlikely \cite{Zu}.  The initial setting is of a Heegaard split $3$-manifold $M = A \cup_T B$ that is the connected sum of compact manifolds $M_1, M_2, M_3$ as shown in Figure \ref{fig:toZupan1}.  The blue indicates the spine $\Sss$ of $B$, say and, following our convention throughout the proof, $B$ is to be thought of as a thin regular neighborhood of $\Sss$.  The spine is not shown inside of the punctured summands $M_1$ and $M_2$ because those parts are irrelevant to the argument; psychologically it's best to think of these as spherical boundary components of $M$ lying in $\bdd_- B$, so $M_1$ and $M_2$ are balls.  

In the figure, $M_3$ is a solid torus and what we see is the punctured $M_3$, lying in $M$ as a summand.  We will continue the argument for this special case, in which $M_3$ is a solid torus and $M_1, M_2$ are balls, but the argument works in general.   An important role is played by the complement $A$ of $\Sss$ outside $M_1$ and $M_2$.  This is a solid torus: indeed, the region in the figure between the torus and the cyan balls is a twice punctured solid torus; $A$ is obtained by removing both a collar of the torus boundary component and the blue arcs, all part of $\Sss$.  Removing the collar does not change the topology, but removing the blue arcs changes the twice-punctured solid torus into an unpunctured solid torus $A$.  

Zupan proposed the following sort of reducing sphere $S$ for $M$: the tube sum of the reducing spheres for $M_1$ and $M_2$ along a tube in $M_3$ which can be arbitrarily complicated.  The outside of the tube is shown in red in Figure \ref{fig:toZupan1}.   The reducing sphere $S$ is not aligned with $T$ because it intersects $\Sss$ in two points, one near each of $M_1$ and $M_2$.  The goal is then to isotope $T$ through $M$ so that it will be aligned with $S$.  This is done by modifying $\Sss$ by what is ultimately a stem swap, and we will describe how the stem swap is obtained by an edge-slide of $\Sss$.  The edge-slide induces an isotopy of $T$ in $M$ because $T$ is the boundary of a regular neighborhood of $\Sss$.  Note that in such an edge slide, passing one of the blue arcs through the red tube is perfectly legitimate.  

\begin{figure}[ht!]
\labellist
\small\hair 2pt
 \pinlabel  $M_1$ at 50 45
  \pinlabel  $M_2$ at 280 45
    \pinlabel  $M_3$ at 250 200
\endlabellist
    \centering
    \includegraphics[scale=0.5]{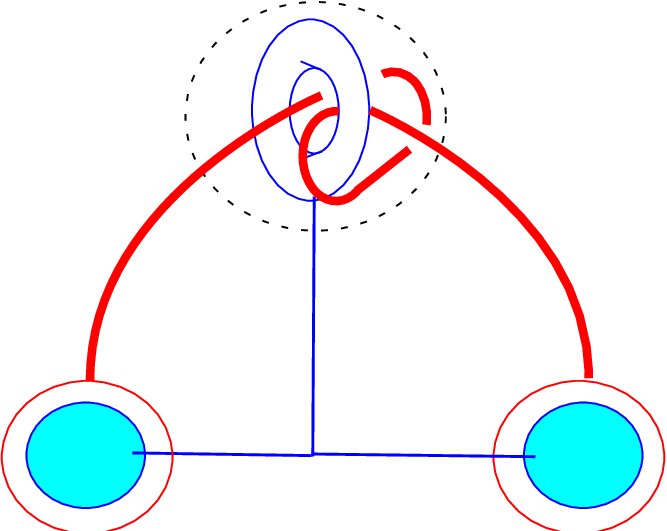}
    \caption{The initial setting}
    \label{fig:toZupan1}
    \end{figure}
    
Figure \ref{fig:toZupan2} is the same, but we have distinguished part of $\Sss$ (the rightmost edge) by turning it teal and beginning to slide it on the rest of the spine.
    
\begin{figure}[ht!]
    \centering
    \includegraphics[scale=0.5]{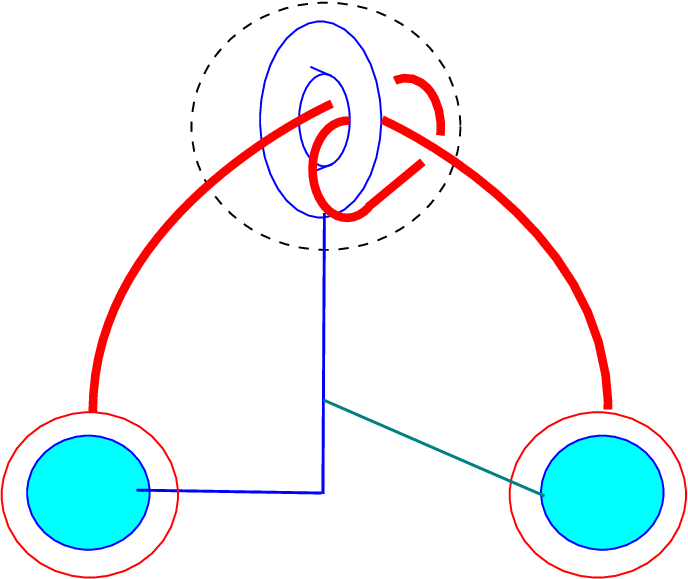}
    \caption{One blue edge now teal}
    \label{fig:toZupan2}
    \end{figure}
    
Now we invoke the viewpoint and notation of Proposition \ref{prop:stemswap}: There is a related Heegaard splitting of $M$ available to us, in which the sphere boundary component at $M_2$ is not viewed as part of $\bdd_- B$ but as part of $\bdd_- A$, and the teal arc is also added to $A$.  This changes $A$ into a punctured solid torus $A_+$ and the spine of its complement into $\Sss_-$, obtained by deleting from $\Sss$ both the teal edge and the sphere boundary component at $M_2$.   
    
\begin{figure}[ht!]
    \centering
    \includegraphics[scale=0.5]{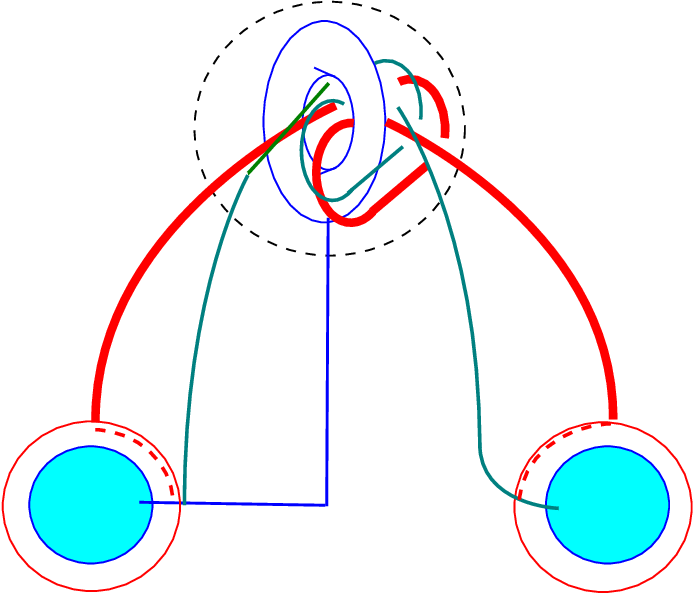}
    \caption{Teal edge now homotopic to red tube}
    \label{fig:toZupan3}
    \end{figure}
    
 And so we apply Lemma \ref{lemma:Zupan}, 
with $A_+$ playing the role of compression-body $C$; the boundary sphere at $M_2$ playing the role of the point $r$; the other end of the teal arc playing the role of $q$; the teal arc playing the role of $\bbb$; and the union of the core of the red tube and the two dotted arcs in Figure \ref{fig:toZupan3} playing the role of $\aaa$.  Specifically, as the proof of Lemma  \ref{lemma:Zupan} describes, because $\pi_1(\bdd A_+) \to \pi_1(A_+)$ is surjective, and the slides take place in $\bdd A_+$, one can slide the end of the teal arc around on the rest of $\Sss_-$ (technically on the boundary of a thin regular neighborhood of $\Sss_-$) until it is {\em homotopic} rel end points to the path that is the union of the core of the tube of $S$ and the two dotted red arcs shown in  Figure \ref{fig:toZupan3}. Hass-Thompson \cite[Proposition 4]{HT} then shows that $\aaa$ and $\bbb$ are isotopic rel end points. 
 
 The result of the isotopy is shown in Figure \ref{fig:toZupan5}; the teal edge now goes right through the tube, never intersecting $S$. Thus $S$ now intersects $\Sss$ in only a single point, near the boundary sphere at $M_1$.  In other words, $S$ is aligned with $T$.   
    
\begin{figure}[th!]
    \centering
    \includegraphics[scale=0.5]{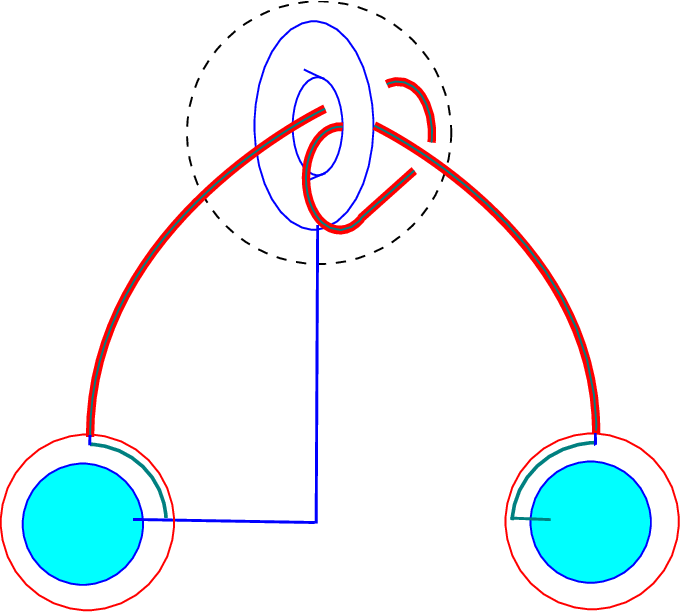}
    \caption{Teal edge isotoped into red tube}
    \label{fig:toZupan5}
    \end{figure}


\begin{thebibliography}{5}

 \bibitem[CG]{CG} A.~Casson and C.~McA.~Gordon, {\em Reducing
Heegaard splittings}, Topology  and its applications, {\bf 27} (1987), 275-283.

 \bibitem[FS1]{FS1}  M.~Freedman, M.~Scharlemann, Dehn's lemma for immersed loops. {\em Math. Res. Lett.  }  {\bf 25} (2018), no. 6, 1827--1836. 
 
 \bibitem[FS2]{FS2}  M.~Freedman, M.~Scharlemann, Uniqueness in Haken's Theorem, arXiv:2004.07385. To appear in {\em Mich. Math. J.}


\bibitem[Ha]{Ha} W.~Haken, {\em Some results on surfaces in 3-manifolds},  
Studies in Modern Topology, Math. Assoc. Am., Prentice Hall, 1968, 34-98.

\bibitem[HS]{HS} S.~Hensel and J.~Schultens, Strong Haken via Sphere Complexes, arXiv:2102.09831.

\bibitem[HT]{HT} J.~Hass, A.~Thompson, Neon bulbs and the unknotting of arcs in manifolds, {\em Journal of Knot Theory and Its Ramifications}  {\bf 6}  (1997) 235-242.

\bibitem[Sc1]{Sc1}  M.~Scharlemann, Heegaard splittings of compact 3-manifolds, {\em Handbook of geometric topology}, ed by R. Daverman and R. Sherr, 921--953, North-Holland, Amsterdam, 2002.

\bibitem[Sc2]{Sc2}  M.~Scharlemann, Generating the Goeritz group, arXiv:2011.10613.

\bibitem[Sch]{Sch} A.~Schoenflies, {\em Beitrage zur Theorie der Punktmengen III}, Math. Ann. {\bf 62} 1906, 286–328

\bibitem[ST]{ST} M.~Scharlemann and A.~Thompson, {\em Thin position and
Heegaard splittings of the 3-sphere}, Jour. Diff. Geom. {\bf 39} 1994, 343-357.

\bibitem[Zu]{Zu} A.~Zupan, personal communication 2019.  


\end{thebibliography}
\end{document}